\documentclass[10pt]{article}
\usepackage[utf8]{inputenc}
\usepackage{amsmath,amsthm,amssymb, mathrsfs,amscd,amsthm,amscd,amsfonts,calligra,mathrsfs,adjustbox,lipsum}
\usepackage[textwidth=360pt,textheight=615pt]{geometry}
\usepackage{indentfirst,url}
\usepackage[all]{xy}
\usepackage{tikz}
\usepackage{paracol}
\usepackage{caption}
\usepackage{float}

\newtheorem{Theorem}{Theorem}
\newtheorem{Proposition}{Proposition}
\newtheorem{Example}{Example}
\newtheorem{Corollary}{Corollary}
\newtheorem{Conjecture}{Conjecture}
\newtheorem{Definition}{Definition}
\newtheorem{Lemma}{Lemma}
\newtheorem{Notation}[Proposition]{Notation}

\DeclareMathOperator{\glct}{glct}
\DeclareMathOperator{\lct}{lct}
\DeclareMathOperator{\mult}{mult}
\DeclareMathOperator{\support}{supp}

\DeclareMathOperator{\Picard}{Pic}

\newtheorem*{Acknowledgments}{ACKNOWLEDGMENTS}

\usepackage{CJKutf8}

\usepackage{amsmath,amssymb,amscd,amsfonts}

\title{Global log canonical thresholds for Burniat surfaces and some infinite families of surfaces of general type as $\mathbb{Z}_2^n$-covers
}
\author{Nguyen Bin, Jheng-Jie Chen and YongJoo Shin}
\date{\today}

\newcommand{\BinAddresses}{{
		\bigskip
		\footnotesize
        \text{Nguyen Bin,}\par\nopagebreak	
		\text{Department of Mathematics and Statistics,}\par\nopagebreak	
		\text{Quy Nhon University,}\par\nopagebreak	
            \text{170 An Duong Vuong Street, }\par\nopagebreak
		\text{Quy Nhon Nam Ward, Gia Lai, Vietnam.}\par\nopagebreak		
		\textit{E-mail address}: \texttt{nguyenbin@qnu.edu.vn}
		
}}

\newcommand{\JhengJieAddresses}{{
		\bigskip
		\footnotesize
        \text{Jheng-Jie Chen,}\par\nopagebreak	
		\text{Department of Mathematics,}\par\nopagebreak	
		\text{National Central University,}\par\nopagebreak	
            \text{No. 300, Zhongda Rd., Zhongli District, }\par\nopagebreak
		\text{Taoyuan City, Taiwan.}\par\nopagebreak		
		\textit{E-mail address}: \texttt{jhengjie@math.ncu.edu.tw}
		
}}

\newcommand{\YongJooAddresses}{{
		\bigskip
		\footnotesize
        \text{YongJoo Shin,}\par\nopagebreak	
		\text{Department of Mathematics,}\par\nopagebreak	
		\text{Chungnam National University,}\par\nopagebreak	
            \text{99, Daehak-ro, Yuseong-gu, }\par\nopagebreak
		\text{Daejeon, Republic of Korea.}\par\nopagebreak		
		\textit{E-mail address}: 
        \texttt{haushin@cnu.ac.kr}
		
}}

\begin{document}

\maketitle
\begin{abstract}
     In this paper, we complete the computation of the global log canonical thresholds of Burniat surfaces. Additionally, we determine the global log canonical thresholds for certain infinite families of surfaces of general type, constructed as $ \mathbb{Z}_2^n $-covers of $\mathbb{P}^2$ or $\mathbb{P}^1 \times \mathbb{P}^1$.

\end{abstract}

\section{Introduction}
    This work aims to calculate the global log canonical thresholds of complex projective algebraic surfaces of general type. While a significant progress has been made in studying the global log canonical thresholds of del Pezzo surfaces and higher-dimensional Fano varieties, a research on surfaces of general type remains relatively sparse. Note that the global log canonical threshold of surfaces of general type with respect to their canonical divisors, and the data of non-klt centers play an important role in the study of threefolds of general type (see \cite{MR4105534, MR2981839,MR3732684} ). Kim and the third author \cite{MR4216579} determined the global log canonical thresholds of Burniat surfaces with $K^2 = 6$. Recently, the authors \cite{2024arXiv240104326B} extended this work to Burniat surfaces with $K^2 = 5$. To the best of our knowledge, these are the only surfaces of general type for which global log canonical thresholds have been computed so far. Inspired by the works of Bauer and Catanese on the theory of Burniat surfaces \cite{MR2609250, zbMATH05885266, zbMATH06220365}, in this paper we complete the computation of the global log canonical thresholds of Burniat surfaces, as summarized below.

    \begin{Theorem}\label{GLCT_Burniatsurfaces_K^2=4}
		Let $ X $ be a Burniat surface. Then 
    \begin{align*}\glct\left( X,K_{X}\right) =\left\{ \begin{array}{ll} 
    \vspace{0.1cm}
    \frac{1}{2}&\textup{if $K_{X}^2 = 6$,}\\ \vspace{0.1cm}
    \frac{1}{2}&\textup{if $K_{X}^2 = 5$,}\\ \vspace{0.1cm}
    \frac{2}{3}&\textup{if $K_{X}^2 = 4$ and $X$ is of non-nodal type,} \\  \vspace{0.1cm} \frac{1}{2}&\textup{if $K_{X}^2 = 4$ and $X$ is of nodal type,} \\ \vspace{0.1cm}  \frac{2}{3}&\textup{if $K_{X}^2 = 3$,}\\ \vspace{0.1cm}  1&\textup{if $K_{X}^2 = 2$.}
    \end{array} \right.\end{align*} 
    \end{Theorem}

    \noindent
    As explained above, to prove Theorem \ref{GLCT_Burniatsurfaces_K^2=4} we compute the global log canonical thresholds of Burniat surfaces with $K^2 = 4,3 \textup{ or } 2$. Our proof strategy follows the approach developed by Kim and the third author \cite{MR4216579, MR4515703}. The equalities stated in Theorem \ref{GLCT_Burniatsurfaces_K^2=4} are derived via a proof by contradiction.  It is noteworthy that the branch loci of Burniat surfaces with $ K^2_X = 3,4$ are more intricate than those with $ K^2_X = 5,6$. Consequently, a greater number of cases must be analyzed. In most instances, contradictions are obtained by applying the inversion of adjunction. However, there exist specific cases in which this method alone does not yield a contradiction directly. To address this difficulty, we employ the blowing-up technique introduced by Cheltsov \cite{MR2465686}. Together, these two approaches enable us to derive the desired contradictions effectively. It is well-known that the Burniat surfaces with $K^2 = 2$ are the classical Campedelli surfaces which are $ \mathbb{Z}^3_2 $-covers of $\mathbb{P}^2$ ramified in $7$ lines (cf. \cite[Section 5]{zbMATH05885266}). The global log canonical thresholds of these surface are computed in Theorem \ref{glct_Z_2^n_covering}.\\

    The theory of abelian covers was studied by Pardini in \cite{MR1103912}. Such coverings are frequently used to construct explicit examples of varieties with interesting properties. We observe that the approach introduced by Kim and the third author can be applied to a wide class of abelian covers. In this paper, we compute the global log canonical thresholds for certain infinite families of surfaces of general type, constructed as $ \mathbb{Z}_2^n $-covers.
    \begin{Theorem}\label{glct_Z_2^n_covering}
    Let $n,k$ be positive integers with $k \geq 7$, and $l_1, l_2, \ldots, l_k$ be $k$ lines in general position in $\mathbb{P}^2$. Let $ \xymatrix{\varphi: X \ar[r] & \mathbb{P}^2}  $ be a $ \mathbb{Z}^n_2 $-cover branched along the divisor $l_1 + l_2 + \cdots + l_k$ such that $X$ is a smooth irreducible surface. Then $$\glct{(X)} = \frac{1}{k-6}.$$
    In particular, $$\glct(X,2K_X)=\alpha_1(X,2K_X)$$ where $\alpha_m$ is $m$th $\alpha$-invariant (see Section \ref{somefacts}).
    \end{Theorem}

    \noindent
    Many known surfaces of general type were constructed as Theorem \ref{glct_Z_2^n_covering} (see Section \ref{application_glct_Z_2^n}). As a consequence of Theorem \ref{glct_Z_2^n_covering}, the Campedelli surfaces $X$ have $\glct(X) = 1$ (see Corollary \ref{glct_Campedelli_surfaces}) and the Persson surfaces $X$ have $\glct(X) = \frac{1}{2}$ (see Corrollary \ref{Persson's generation}). 

    \begin{Theorem}\label{glct_Z_2^n_covering_P^1xP^1}
    Let $n,k,k'$ be positive integers with $k'\geq k \geq 5$, and $F_1, F_2, \ldots, F_k \in | F |$, $G_1, G_2, \ldots, G_{k'} \in | G |$ be distinct $k+k'$ divisors in $\mathbb{P}^1 \times \mathbb{P}^1$ where $F = \{0\}\times \mathbb{P}^1$ and $G = \mathbb{P}^1 \times \{0\}$. Let $ \xymatrix{\varphi: X \ar[r] & \mathbb{P}^1 \times \mathbb{P}^1}  $ be a $ \mathbb{Z}^n_2 $-cover branched along the divisor $\sum \limits_{i=1}^{k}{F_i} + \sum \limits_{j=1}^{k'}{G_j}$  such that $X$ is a smooth irreducible surface. Then $$\glct{(X)} = \frac{1}{k'-4}.$$
    In particular, $$\glct(X,2K_X)=\alpha_1(X,2K_X).$$
    \end{Theorem}
    
    \noindent
    Many known irregular surfaces of general type were constructed as Theorem  \ref{glct_Z_2^n_covering_P^1xP^1}. See Examples \ref{Example 1} and \ref{Example 2} for explicit information.

    Throughout this paper, all surfaces are projective algebraic over a field of complex numbers. The linear equivalence of the divisors is indicated by $ \sim $ and the $\mathbb{Q}$-linear equivalence is indicated by $\sim_{\mathbb{Q}}$.  We denote by $F = \{0\}\times \mathbb{P}^1$ and $G = \mathbb{P}^1 \times \{0\}$ the generators of $\Picard(\mathbb{P}^1 \times \mathbb{P}^1)$. And $n$ is a positive integer.

    \section{Preliminaries}

        \subsection{$ \mathbb{Z}^n_2 $-covers}
The construction of abelian covers was studied by Pardini in \cite{MR1103912}. For details of the building data of abelian covers and their notation, we refer the reader to Sections 1 and 2 of Pardini's work \cite{MR1103912}. For a deeper discussion of $ \mathbb{Z}^n_2 $-covers, we refer the reader to \cite[Section 3]{MR1802792}. For the sake of completeness, we recall some facts on $ \mathbb{Z}_{2}^n $-covers, in a form that is convenient for our later proofs.
	       
	       We denote by  $ \chi_{j_1j_2\ldots j_n} $ the character of $ \mathbb{Z}_{2}^n $ defined by
	       \begin{align*}
	       	\chi_{j_1j_2\ldots j_n}\left( a_1,a_2,\ldots,a_n\right): =  e^{\left( \pi a_1j_1\right) \sqrt{-1}}e^{\left( \pi a_2j_2\right) \sqrt{-1}} \ldots e^{\left( \pi a_nj_n\right) \sqrt{-1}}
	       \end{align*}
	       for all $ j_1,j_2, \ldots, j_n,a_1,a_2,\ldots, a_n\in \mathbb{Z}_2 $. For a normal surface $X$ and a smooth complete surface $Y$, a $ \mathbb{Z}_{2}^n $-cover $ \xymatrix{X \ar[r] & Y} $ can be determined by a collection of non-trivial divisors $ L_{\chi} $ labelled by characters of $ \mathbb{Z}_{2}^n $ and effective divisors $ D_{\sigma} $ labelled by non-trivial elements of $ \mathbb{Z}_{2}^n $ of the surface $ Y $. More precisely, we have the result for $\mathbb{Z}_{2}^n$-covers as follows:
	       \begin{Proposition}[{cf. \cite[\rm Theorem 2.1]{MR1103912}}]\label{Construction of Z_2^n covers}
	       	Given a smooth projective surface $Y$, let $ L_{\chi} $ be divisors of $ Y $ such that $ L_{\chi} \not\sim \mathcal{O}_Y $ for all non-trivial characters $ \chi $ of $ \mathbb{Z}_{2}^n  $, and $ D_{\sigma} $ be effective divisors of  $ Y $ for all $ \sigma \in \mathbb{Z}_{2}^n \setminus \left\lbrace \left(0,0,\ldots, 0 \right)  \right\rbrace  $ such that the total branch divisor $ B:=\sum\limits_{\sigma \ne 0}{D_{\sigma}} $ is reduced. Then $ \left\lbrace L_{\chi}, D_{\sigma} \right\rbrace_{\chi,\sigma}$ is the building data of a $ \mathbb{Z}_{2}^n$-cover $ \xymatrix{f:X \ar[r]& Y} $ for a normal surface $X$ if and only if
	       	\begin{align}\label{The condition of Z_2^n covers} 
	       		L_{\chi}+L_{\chi'} \sim L_{\chi\chi'}+ \sum\limits_{\chi\left( \sigma\right)=\chi'\left( \sigma\right) = -1 }{D_{\sigma}}	
	       	\end{align}
	       	for all non-trivial characters $ \chi, \chi' $ of $ \mathbb{Z}_{2}^n  $.

	       \end{Proposition} 

            \noindent
	    We notice that if $\Picard(Y)$ has no $2$-torsion, then the branch data $\{D_{\sigma} \}_{\sigma}$ suffices to determine the cover.  By \cite[\rm Proposition 3.1]{MR1103912} if each branch component $D_\sigma$ is smooth and the total branch locus $B $ is a simple normal crossings divisor, the surface $X$ is smooth. \\
	       
	       \noindent
	       Also from \cite[Remark 3.11]{MR1802792} and \cite[\rm Propositions 4.1 and 4.2]{MR1103912} we have:
	       \begin{Proposition}        \label{invariants of Z_2^n covers}
	       	If $ Y $ is a smooth surface and $ \xymatrix{f: X \ar[r]& Y} $ is a smooth $  \mathbb{Z}_{2}^n$-cover with the building data $ \left\lbrace L_{\chi}, D_{\sigma} \right\rbrace_{\chi,\sigma}$, the surface $ X $ satisfies the following:
	       	\begin{align*}
	       		2K_X & \sim f^*\left( 2K_Y + \sum\limits_{\sigma \ne 0} {D_{\sigma} } \right); \\
	       		f_{*}\mathcal{O}_X &= \mathcal{O}_Y \oplus \bigoplus\limits_{\chi \ne \chi_{00\ldots0}  }L_{\chi}^{-1}.
	       	\end{align*}	
	       	\noindent
	       	This implies that	
	       	\begin{align*}	       		
	       		K^2_X &= 2^{n-2}\left( 2K_Y + \sum\limits_{\sigma \ne 0} {D_{\sigma} } \right)^2; \\
	       		p_g\left( X \right) &=p_g\left( Y \right) +\sum\limits_{\chi \ne  \chi_{00\dots0}  }{h^0\left( Y, \mathcal{O}_Y (L_{\chi} + K_Y) \right)}; \\
	       		\chi\left( \mathcal{O}_X \right) &= 2^{n}\chi\left( \mathcal{O}_Y \right)  +\sum\limits_{\chi \ne \chi_{00\ldots0}  }{\frac{1}{2}L_{\chi}\left( L_{\chi}+K_Y\right)}. 
	       	\end{align*}
	       		       
	       \end{Proposition}

    \subsection{Some facts on global log canonical thresholds}\label{somefacts}
    The following is taken from \cite{MR2465686, MR4216579}. The standard reference is \cite{MR1492525}.
    \begin{Definition}
       Let $X$ be an algebraic surface with log terminal singularities, $Z \subseteq X$ be a closed subvariety, and let $D$ be an effective $\mathbb{Q}$-Cartier on $X$. The number 
       $$ \lct_{Z}\left( X, D\right) :=\sup\left\lbrace \lambda \in \mathbb{Q} \mid \text{the log pair} \left( X,\lambda D\right) \text{is log canonical along } Z  \right\rbrace  $$
       is said to be the log canonical threshold of $D$ along $Z$. If $Z = X$, then we write $\lct(X,D) = \lct_{Z}(X,D)$.
       
    \end{Definition}

    \begin{Definition}
       Let $X$ be an algebraic surface with log terminal singularities. The global log canonical threshold of  $X$ with respect to a $\mathbb{Q}$-Cartier divisor $L$, denoted by $\glct(X, L)$, is defined as the number 
       \begin{equation*}
	    \glct\left( X, L\right) :=\inf\left\lbrace \lct\left( X, D\right) \begin{array}{|l} 
         D \text{ is an effective } \mathbb{Q} \text{-Cartier divisor on } X,\\
        \mathbb{Q}\text{-linearly equivalent to } L. 
        \end{array}
       \!\!\! \right\rbrace\!.
	\end{equation*}
    If $L=K_X$, then we denote this as $\glct\left( X\right) = \glct\left( X, K_X\right)$.
       
    \end{Definition}

    While the value of $\glct(X, L)$ is often difficult to compute directly, it can sometimes be approximated using more tractable invariants. For example, when the linear system $|mL|$ for some positive integer is non-empty, Tian introduced the \emph{$m$-th $\alpha$-invariant} of the pair $(X, L)$ (cf. \cite{zbMATH06851625}), defined by
    \begin{equation*}
    \alpha_m(X, L) := \sup\left\{ \lambda \in \mathbb{Q} \,\middle|\, \text{the pair } \left(X, \tfrac{\lambda}{m} D\right) \text{ is log canonical for every } D \in |mL| \right\}.
    \end{equation*}
    \noindent
    If the linear system $|mL|$ is empty, we simply set $\alpha_m(X, L) = + \infty$. Then we have the inequality $\glct \left(X,L \right) \le \alpha_m(X, L)$. Tian subsequently proposed the following conjecture.
        \begin{Conjecture}[{\cite[Conjecture 5.4]{zbMATH06072915}}]
            Suppose that $L$ is very ample and defines a projectively normal embedding under its associated morphism, that is, the graded algebra
            $$\bigoplus\limits_{i \ge 0  }H^{0}\left(X, \mathcal{O}_{X}\left(iL \right) \right)$$
            is generated by elements in $H^{0}\left(X, \mathcal{O}_{X}\left(L \right) \right)$. Then $\glct \left(X,L \right) = \alpha_1(X, L)$.
        \end{Conjecture}

    We use the following Lemmas into Sections \ref{glctinf} and \ref{glctbur}.
    \begin{Lemma}\label{comarelct}
   		Suppose $L_1$ and $L_2$ are two distinct lines in $\mathbb{A}^2$. Let $n_1$ and $n_2$ be two integers with $n_2\geq n_1>0$. Then $\lct(\mathbb{A}^2, n_1L_1+n_2L_2)=\frac{1}{n_2}$.
   \end{Lemma}
     Lemma \ref{comarelct} follows from \cite[Example 8.17]{MR1492525} or \cite[Proposition 2.2]{MR1704476}.\\

    \begin{Notation}
       Let $\psi\colon X\to Y$ be a finite morphism. If a $\mathbb{Q}$-divisor
    $D$ on $X$ is of the form $D=\psi^*d$ for a $\mathbb{Q}$-divisor $d$ on $Y$,
    we write
    \[
    d=\psi(D).
    \] 
    \end{Notation}
 
    \noindent We present a slight generalization of \cite[Lemma 4.1]{MR4216579} of Kim and the third author.
    \begin{Lemma}\label{log_canonical_between_cover}
    Let $\xymatrix{f: X \ar[r] & Y}$ be a $ \mathbb{Z}^n_2 $-cover between a normal surface $X$ and a smooth surface $Y$ branched along an effective divisor $B$ on $Y$, and $D$ be an effective $\mathbb{Q}$-Cartier divisor on $X$. Then the pair $(Y, f(D) + \frac{1}{2}B)$ is not log canonical at $f(P)$ if the pair $(X,D)$ is not log canonical at a point $P$ on $X$.
\end{Lemma}
\begin{proof}
    Suppose that the pair $(X, D)$ is not log canonical at some $P \in X$. Since $f^{*}(f(D)) \ge D$, the pair $(X, f^{*}(f(D)))$ is not log canonical at $P $. By Proposition \ref{invariants of Z_2^n covers}, we get $K_X \sim_{\mathbb{Q}} f^{*}(K_Y + \frac{1}{2}B)$. So 
    $$K_X + f^{*}(f(D)) \sim_{\mathbb{Q}} f^{*}(K_Y + \frac{1}{2}B + f(D)).$$
    Thus, by \cite[Proposition 3.16]{MR1492525}, the pair $(X, f^{*}(f(D)))$ is not log canonical if and only if the pair $(Y, f(D) + \frac{1}{2}B)$ is not log canonical. Therefore, the pair $(Y, f(D) + \frac{1}{2}B)$ is not log canonical.
\end{proof}

    \subsection{Constructions of Burniat Surfaces with $K^2=4$ and $3$.}
    The standard work on the Burniat surfaces with $K_X^2=4,3$ is \cite{MR2609250}. For the sake of convenience, in this section we present the construction of these Burniat surfaces.

    \subsubsection{Burniat surfaces with $K^2 = 4$}
    \paragraph{$K^2 = 4 $ non-nodal type}\label{Construction_of_Burniat_surface_with_K^2=4_non_nodal_type}
    \begin{Notation}\label{Notation of del Pezzo surface of degree 5}
	We denote by $ Y_5 $ the blow-up of $ \mathbb{P}^2$ at five points in general position $ P_1, P_2, P_3, P_4, P_5 $. Let us denote by $ l $ the pull-back of a general line in $ \mathbb{P}^2$, by $ e_1 $, $ e_2 $, $ e_3 $, $ e_4 $, $ e_5 $ the exceptional divisors corresponding to $ P_1 $, $ P_2 $, $ P_3 $, $ P_4 $, $ P_5 $, respectively, by $ t_1 $, $ t_2$, $ t_3$, $ t_4$, $ t_5$ the strict transforms of a general line through $ P_1 $, $ P_2 $, $ P_3 $, $ P_4 $, $ P_5 $, respectively and by $ h_{ij} $ the strict transforms of the line $ P_i  P_j $, for all $ i, j \in \left\lbrace 1,2,3,4,5\right\rbrace  $, respectively. The anti-canonical class 
	\begin{align*}
	-K_{Y_5} \sim 3l-e_1 - e_2 - e_3 - e_4 - e_5.
	\end{align*}	   	   		
	\end{Notation}
 
\noindent
 We consider the following divisors of $ Y_5 $ 
{\small
    \begin{align*}
		B_{1}&:= e_1 + h_{23} + h_{24} + h_{25}  \sim 3l+e_1 - 3e_2 - e_3 - e_4- e_5, &L_{1}&: = 3l-2e_1-e_3 - e_4- e_5\\
		B_{2}&:= e_2 + h_{13} + h_{34} + h_{35}  \sim 3l-e_1 +e_2 - 3e_3 - e_4- e_5, &L_{2}&: = 3l-e_1 -2e_2 - e_4- e_5\\
		B_{3}&:= e_3 + h_{12} + h_{14} + h_{15}  \sim 3l-3e_1 -e_2 +e_3 - e_4- e_5, &L_{3}&: = 3l - e_2 -2e_3 - e_4- e_5.	
	\end{align*}
}
 \noindent

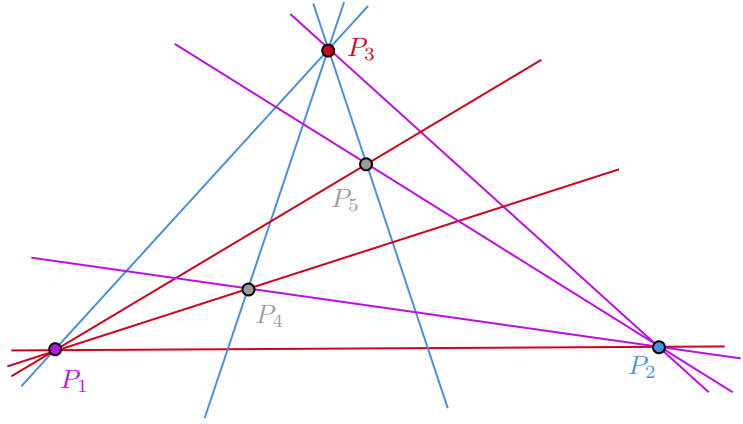
\begin{figure}[H]
\begin{center}

\tikzset{every picture/.style={line width=0.75pt}} 

\begin{tikzpicture}[x=0.75pt,y=0.75pt,yscale=-1,xscale=1]

\draw [color={rgb, 255:red, 74; green, 144; blue, 226 }  ,draw opacity=1 ]   (307,31.6) -- (133,222.6) ;
\draw [color={rgb, 255:red, 208; green, 2; blue, 27 }  ,draw opacity=1 ]   (128,204.6) -- (486,202.6) ;
\draw [color={rgb, 255:red, 189; green, 16; blue, 224 }  ,draw opacity=1 ]   (268,35.6) -- (478,225.6) ;
\draw [color={rgb, 255:red, 74; green, 144; blue, 226 }  ,draw opacity=1 ]   (295,29.6) -- (225,238.6) ;
\draw [color={rgb, 255:red, 74; green, 144; blue, 226 }  ,draw opacity=1 ]   (279.5,30.3) -- (347,233.6) ;
\draw [color={rgb, 255:red, 208; green, 2; blue, 27 }  ,draw opacity=1 ]   (126,212.6) -- (433,113.6) ;
\draw [color={rgb, 255:red, 208; green, 2; blue, 27 }  ,draw opacity=1 ]   (128,217.6) -- (394,58.6) ;
\draw [color={rgb, 255:red, 189; green, 16; blue, 224 }  ,draw opacity=1 ]   (138,158) -- (494,208.6) ;
\draw [color={rgb, 255:red, 189; green, 16; blue, 224 }  ,draw opacity=1 ]   (210,50.6) -- (490,225.6) ;
\draw  [fill={rgb, 255:red, 208; green, 2; blue, 27 }  ,fill opacity=1 ] (284,54) .. controls (284,52.34) and (285.34,51) .. (287,51) .. controls (288.66,51) and (290,52.34) .. (290,54) .. controls (290,55.66) and (288.66,57) .. (287,57) .. controls (285.34,57) and (284,55.66) .. (284,54) -- cycle ;
\draw  [fill={rgb, 255:red, 74; green, 144; blue, 226 }  ,fill opacity=1 ] (450,203) .. controls (450,201.34) and (451.34,200) .. (453,200) .. controls (454.66,200) and (456,201.34) .. (456,203) .. controls (456,204.66) and (454.66,206) .. (453,206) .. controls (451.34,206) and (450,204.66) .. (450,203) -- cycle ;
\draw  [color={rgb, 255:red, 0; green, 0; blue, 0 }  ,draw opacity=1 ][fill={rgb, 255:red, 189; green, 16; blue, 224 }  ,fill opacity=1 ] (147,204) .. controls (147,202.34) and (148.34,201) .. (150,201) .. controls (151.66,201) and (153,202.34) .. (153,204) .. controls (153,205.66) and (151.66,207) .. (150,207) .. controls (148.34,207) and (147,205.66) .. (147,204) -- cycle ;
\draw  [fill={rgb, 255:red, 155; green, 155; blue, 155 }  ,fill opacity=1 ] (303,111.1) .. controls (303,109.44) and (304.34,108.1) .. (306,108.1) .. controls (307.66,108.1) and (309,109.44) .. (309,111.1) .. controls (309,112.76) and (307.66,114.1) .. (306,114.1) .. controls (304.34,114.1) and (303,112.76) .. (303,111.1) -- cycle ;
\draw  [fill={rgb, 255:red, 155; green, 155; blue, 155 }  ,fill opacity=1 ] (244,174) .. controls (244,172.34) and (245.34,171) .. (247,171) .. controls (248.66,171) and (250,172.34) .. (250,174) .. controls (250,175.66) and (248.66,177) .. (247,177) .. controls (245.34,177) and (244,175.66) .. (244,174) -- cycle ;

\draw (151,213) node [anchor=north west][inner sep=0.75pt]  [color={rgb, 255:red, 189; green, 16; blue, 224 }  ,opacity=1 ] [align=left] {$P_1$};
\draw (436,206) node [anchor=north west][inner sep=0.75pt]  [color={rgb, 255:red, 74; green, 144; blue, 226 }  ,opacity=1 ] [align=left] {$P_2$};
\draw (295,46) node [anchor=north west][inner sep=0.75pt]  [color={rgb, 255:red, 208; green, 2; blue, 27 }  ,opacity=1 ] [align=left] {$P_3$};
\draw (249,180) node [anchor=north west][inner sep=0.75pt]  [color={rgb, 255:red, 155; green, 155; blue, 155 }  ,opacity=1 ] [align=left] {$P_4$};
\draw (287,122) node [anchor=north west][inner sep=0.75pt]  [color={rgb, 255:red, 155; green, 155; blue, 155 }  ,opacity=1 ] [align=left] {$P_5$};

\end{tikzpicture}

\caption{Schematic depiction of the branch locus.}
\end{center}
\end{figure}
 
	The divisors $ B_{1}, B_{2}, B_{3}, L_{1}, L_{2}, L_{3} $ define a $ \mathbb{Z}^2_2 $-cover $ \xymatrix{\varphi: X \ar[r] & Y_5}  $ branched along $B_{1}+B_{2}+B_{3}$. The surface $ X $ satisfies the following:		
	\begin{align*}
		2K_{X} &\sim \varphi^*\left( 3l-e_1 - e_2 - e_3 - e_4 - e_5\right) \sim \varphi^*\left( -K_{Y_5}\right).
	\end{align*}
	
    Since the divisor $ 2K_{X} $ is the pull-back of a nef and big divisor, the canonical divisor $ K_{X} $ is nef and big. Thus, the surface $ X $ is of general type and minimal. Furthermore, the surface $X$ possesses the following invariants:
	\begin{align*}
		K_{X}^2=& \left( 3l-e_1 - e_2 - e_3 - e_4 - e_5\right)^2=4,\\
		p_g\left( X\right) =&p_g\left( Y_5 \right) + \sum_{i=1}^3 h^0\left(Y_5, \mathcal{O}_{Y_5}(L_{i} + K_{Y_5}) \right)=0,\\
		\chi\left( \mathcal{O}_{X}\right) =&4\chi\left( \mathcal{O}_{Y_5} \right)  + \sum_{i=1}^3 \frac{1}{2}L_{i}\left( L_{i}+K_{Y_5}\right)=1,\\
		q\left( X\right)  =& 1+p_g\left( X\right)-\chi\left( \mathcal{O}_{X}\right)=0.
	\end{align*}

    \begin{Notation}
        For all $i = 1,2,3$, we denote $E_i := \varphi^{*}(e_i)_{\text{red}}$. Furthermore, we set $\tilde{E}_4 := \varphi^{*}(e_4)$, $\tilde{E}_5 := \varphi^{*}(e_5)$, $\tilde{H}_{45} := \varphi^{*}(h_{45})$, and $H_{ij} := \varphi^{*}(h_{ij})_{\text{red}}$ for all distinct $i, j \in \{1,2,3,4,5\}$ with $(i,j) \neq (4,5)$.
    \end{Notation}

    \paragraph{$K^2 = 4 $ nodal type}\label{Construction_of_Burniat_surface_with_K^2=4_nodal_type}

    \begin{Notation}\label{Notation of Y_4}
        Let $P_1, P_2, P_3, P_4$ be four points in general position in $\mathbb{P}^2$ and $P_5$ be a general point in the line $P_1P_4$. We denote by $ Y_4 $ the blow-up of $ \mathbb{P}^2$ at four points $ P_1, P_2, P_3, P_4$.
    \end{Notation}
    
    \begin{Notation}\label{Notation of weak del Pezzo surface of degree 5}
	We denote by $ Y_{5}' $ the blow-up of $ \mathbb{P}^2$ at five points $ P_1, P_2, P_3, P_4, P_5 $ in Notation \ref{Notation of Y_4}. Let us denote by $ l $ the pull-back of a general line in $ \mathbb{P}^2$, by $ e_1 $, $ e_2 $, $ e_3 $, $ e_4 $, $ e_5 $ the exceptional divisors corresponding to $ P_1 $, $ P_2 $, $ P_3 $, $ P_4 $, $ P_5 $, respectively, by $ t_1 $, $ t_2$, $ t_3$, $ t_4$, $ t_5$ the strict transforms of a general line through $ P_1 $, $ P_2 $, $ P_3 $, $ P_4 $, $ P_5 $, respectively, by $ h_{ij} $ the strict transforms of the line $ P_i  P_j $, for all $ i, j \in \left\lbrace 1,2,3,4,5\right\rbrace  $, respectively except the cases $ i, j \in \left\lbrace 1,4,5\right\rbrace  $, and by $ h_{145} $ the strict transforms of the line $ P_1  P_4 P_5$. The anti-canonical class 
	\begin{align*}
	-K_{Y_{5}'} \sim 3l-e_1 - e_2 - e_3 - e_4 - e_5.
	\end{align*}	   	   		
	\end{Notation}

    \noindent
     We consider the following divisors of $ Y_{5}' $ 
{\small
    \begin{align*}
		B_{1}&:= e_1 + h_{23} + h_{24} + h_{25}  \sim 3l+e_1 - 3e_2 - e_3 - e_4- e_5, &L_{1}&: = 3l-2e_1-e_3 - e_4- e_5\\
		B_{2}&:= e_2 + h_{13} + h_{34} + h_{35}  \sim 3l-e_1 +e_2 - 3e_3 - e_4- e_5, &L_{2}&: = 3l-e_1 -2e_2 - e_4- e_5\\
		B_{3}&:= e_3 + h_{12} + h_{145} + t_{11}  \sim 3l-3e_1 -e_2 +e_3 - e_4- e_5, &L_{3}&: = 3l - e_2 -2e_3 - e_4- e_5,	
	\end{align*}
}
\noindent
where $ t_{11} \in \left| t_1\right| $ is a divisor of $ Y_{5}' $ such that $ B_{1} + B_{2} + B_{3} $ is simple normal crossing. 

    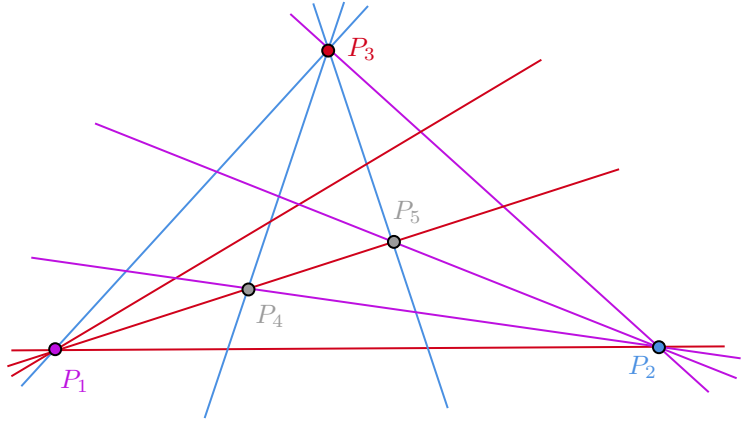
\begin{figure}[H]

    \begin{center}

\tikzset{every picture/.style={line width=0.75pt}} 

\begin{tikzpicture}[x=0.75pt,y=0.75pt,yscale=-1,xscale=1]

\draw [color={rgb, 255:red, 74; green, 144; blue, 226 }  ,draw opacity=1 ]   (307,31.6) -- (133,222.6) ;
\draw [color={rgb, 255:red, 208; green, 2; blue, 27 }  ,draw opacity=1 ]   (128,204.6) -- (486,202.6) ;
\draw [color={rgb, 255:red, 189; green, 16; blue, 224 }  ,draw opacity=1 ]   (268,35.6) -- (478,225.6) ;
\draw [color={rgb, 255:red, 74; green, 144; blue, 226 }  ,draw opacity=1 ]   (295,29.6) -- (225,238.6) ;
\draw [color={rgb, 255:red, 74; green, 144; blue, 226 }  ,draw opacity=1 ]   (279.5,30.3) -- (347,233.6) ;
\draw [color={rgb, 255:red, 208; green, 2; blue, 27 }  ,draw opacity=1 ]   (126,212.6) -- (433,113.6) ;
\draw [color={rgb, 255:red, 208; green, 2; blue, 27 }  ,draw opacity=1 ]   (128,217.6) -- (394,58.6) ;
\draw [color={rgb, 255:red, 189; green, 16; blue, 224 }  ,draw opacity=1 ]   (138,158) -- (494,208.6) ;
\draw [color={rgb, 255:red, 189; green, 16; blue, 224 }  ,draw opacity=1 ]   (170,90.6) -- (492,218.6) ;
\draw  [fill={rgb, 255:red, 208; green, 2; blue, 27 }  ,fill opacity=1 ] (284,54) .. controls (284,52.34) and (285.34,51) .. (287,51) .. controls (288.66,51) and (290,52.34) .. (290,54) .. controls (290,55.66) and (288.66,57) .. (287,57) .. controls (285.34,57) and (284,55.66) .. (284,54) -- cycle ;
\draw  [fill={rgb, 255:red, 74; green, 144; blue, 226 }  ,fill opacity=1 ] (450,203) .. controls (450,201.34) and (451.34,200) .. (453,200) .. controls (454.66,200) and (456,201.34) .. (456,203) .. controls (456,204.66) and (454.66,206) .. (453,206) .. controls (451.34,206) and (450,204.66) .. (450,203) -- cycle ;
\draw  [color={rgb, 255:red, 0; green, 0; blue, 0 }  ,draw opacity=1 ][fill={rgb, 255:red, 189; green, 16; blue, 224 }  ,fill opacity=1 ] (147,204) .. controls (147,202.34) and (148.34,201) .. (150,201) .. controls (151.66,201) and (153,202.34) .. (153,204) .. controls (153,205.66) and (151.66,207) .. (150,207) .. controls (148.34,207) and (147,205.66) .. (147,204) -- cycle ;
\draw  [fill={rgb, 255:red, 155; green, 155; blue, 155 }  ,fill opacity=1 ] (317,150.1) .. controls (317,148.44) and (318.34,147.1) .. (320,147.1) .. controls (321.66,147.1) and (323,148.44) .. (323,150.1) .. controls (323,151.76) and (321.66,153.1) .. (320,153.1) .. controls (318.34,153.1) and (317,151.76) .. (317,150.1) -- cycle ;
\draw  [fill={rgb, 255:red, 155; green, 155; blue, 155 }  ,fill opacity=1 ] (244,174) .. controls (244,172.34) and (245.34,171) .. (247,171) .. controls (248.66,171) and (250,172.34) .. (250,174) .. controls (250,175.66) and (248.66,177) .. (247,177) .. controls (245.34,177) and (244,175.66) .. (244,174) -- cycle ;

\draw (151,213) node [anchor=north west][inner sep=0.75pt]  [color={rgb, 255:red, 189; green, 16; blue, 224 }  ,opacity=1 ] [align=left] {$P_1$};
\draw (436,206) node [anchor=north west][inner sep=0.75pt]  [color={rgb, 255:red, 74; green, 144; blue, 226 }  ,opacity=1 ] [align=left] {$P_2$};
\draw (295,46) node [anchor=north west][inner sep=0.75pt]  [color={rgb, 255:red, 208; green, 2; blue, 27 }  ,opacity=1 ] [align=left] {$P_3$};
\draw (249,180) node [anchor=north west][inner sep=0.75pt]  [color={rgb, 255:red, 155; green, 155; blue, 155 }  ,opacity=1 ] [align=left] {$P_4$};
\draw (318,128) node [anchor=north west][inner sep=0.75pt]  [color={rgb, 255:red, 155; green, 155; blue, 155 }  ,opacity=1 ] [align=left] {$P_5$};

\end{tikzpicture}
\caption{Schematic depiction of the branch locus.}
    \end{center}
\end{figure}

    \noindent
    The divisors $ B_{1}, B_{2}, B_{3}, L_{1}, L_{2}, L_{3} $ define a $ \mathbb{Z}^2_2 $-cover $ \xymatrix{\varphi: X \ar[r] & Y_{5}'}  $ branched along $B_{1}+B_{2}+B_{3}$. The surface $ X $ satisfies the following:		
	\begin{align*}
		2K_{X} &\sim \varphi^*\left( 3l-e_1 - e_2 - e_3 - e_4 - e_5\right) \sim \varphi^*\left( -K_{Y_{5}'}\right).
	\end{align*}

   \noindent 
    Furthermore, the surface $ X $ is of general type and minimal and possesses the following invariants:
    \[K_{X}^2=4, p_g=0, \chi\left( \mathcal{O}_{X}\right)=1, \textup{ and }q(X)=0. \]

    \begin{Notation}
    We denote by $E_i := \varphi^*(e_i)_{\text{red}}$ for $i = 1, 2, 3$, 
    $\tilde{E}_4 := \varphi^*(e_4)$, 
    $\tilde{E}_5 := \varphi^*(e_5)$, $T_{11} := \varphi^*(t_{11})_{\text{red}}$,
    $H_{145} := \varphi^*(h_{145})_{\text{red}}$, and 
    $H_{ij} := \varphi^*(h_{ij})_{\text{red}}$ for all other pairs $1 \le i < j \le 5$.
    \end{Notation}

    \subsubsection{Burniat surfaces with $K^2 = 3$}\label{Construction_of_Burniat_surface_with_K^2=3}
    Let $P_1, P_2, P_3, P_4$ be four points in general position in $\mathbb{P}^2$, $P_5$ be a general point in the line $P_1P_4$, and $P_6$ be the intersection between $P_2P_4$ and $P_3P_5$.

    \begin{Notation}\label{Notation of weak del Pezzo surface of degree 3}
	We denote by $ Y_6 $ the blow-up of $ \mathbb{P}^2$ at six points $ P_1, P_2, P_3, P_4, P_5, P_6 $. Let us denote by $ l $ the pull-back of a general line in $ \mathbb{P}^2$, by $ e_1 $, $ e_2 $, $ e_3 $, $ e_4 $, $ e_5 $, $ e_6 $ the exceptional divisors corresponding to $ P_1 $, $ P_2 $, $ P_3 $, $ P_4 $, $ P_5 $, $ P_6 $, respectively, by $ t_1 $, $ t_2$, $ t_3$, $ t_4$, $ t_5$, $ t_6$ the strict transforms of a general line through $ P_1 $, $ P_2 $, $ P_3 $, $ P_4 $, $ P_5 $, $ P_6 $, respectively, by $ h_{ij} $ the strict transforms of the line $ P_i  P_j $, for all $ i, j \in \left\lbrace 1,2,3,4,5,6\right\rbrace  $, respectively except the cases $ \{i, j\} \subset \left\lbrace 1,4,5\right\rbrace, \left\lbrace 2,4,6\right\rbrace$ or $\left\lbrace 3,5,6\right\rbrace$, and by $ h_{145} $, $ h_{246} $, $ h_{356} $ the strict transforms of the line $ P_1  P_4 P_5$, $ P_2  P_4 P_6$, $ P_3  P_5 P_6$, respectively. The anti-canonical class 
	\begin{align*}
	-K_{Y_6} \sim 3l-e_1 - e_2 - e_3 - e_4 - e_5 -e_6.
	\end{align*}	   	   		
	\end{Notation}

    We consider the following divisors of $ Y_6 $ 
{\small
    \begin{align*}
		B_{1}&:= e_1 + h_{23} + h_{246} + h_{25}  \sim 3l+e_1 - 3e_2 - e_3 - e_4- e_5- e_6, &L_{1}&: = 3l-2e_1-e_3 - e_4- e_5- e_6\\
		B_{2}&:= e_2 + h_{13} + h_{34} + h_{356}  \sim 3l-e_1 +e_2 - 3e_3 - e_4- e_5- e_6, &L_{2}&: = 3l-e_1 -2e_2 - e_4- e_5- e_6\\
		B_{3}&:= e_3 + h_{12} + h_{145} + h_{16}  \sim 3l-3e_1 -e_2 +e_3 - e_4- e_5- e_6, &L_{3}&: = 3l - e_2 -2e_3 - e_4- e_5- e_6.	
	\end{align*}
}
    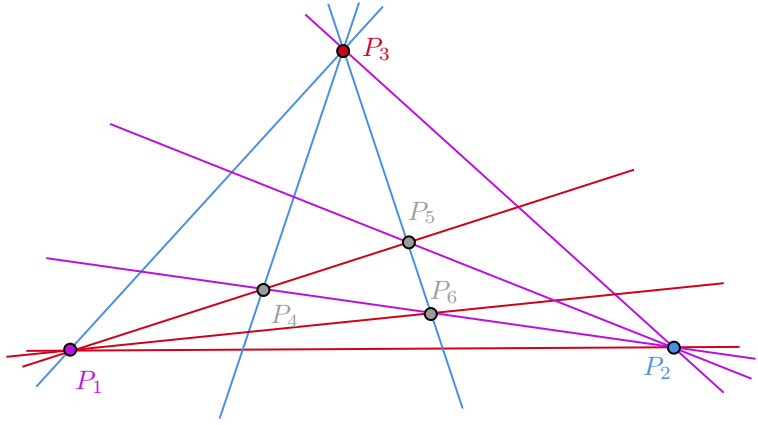
\begin{figure}[H]  
    \begin{center}

\tikzset{every picture/.style={line width=0.75pt}} 

\begin{tikzpicture}[x=0.75pt,y=0.75pt,yscale=-1,xscale=1]

\draw [color={rgb, 255:red, 74; green, 144; blue, 226 }  ,draw opacity=1 ]   (307,31.6) -- (133,222.6) ;
\draw [color={rgb, 255:red, 208; green, 2; blue, 27 }  ,draw opacity=1 ]   (128,204.6) -- (486,202.6) ;
\draw [color={rgb, 255:red, 189; green, 16; blue, 224 }  ,draw opacity=1 ]   (268,35.6) -- (478,225.6) ;
\draw [color={rgb, 255:red, 74; green, 144; blue, 226 }  ,draw opacity=1 ]   (295,29.6) -- (225,238.6) ;
\draw [color={rgb, 255:red, 74; green, 144; blue, 226 }  ,draw opacity=1 ]   (279.5,30.3) -- (347,233.6) ;
\draw [color={rgb, 255:red, 208; green, 2; blue, 27 }  ,draw opacity=1 ]   (126,212.6) -- (433,113.6) ;
\draw [color={rgb, 255:red, 208; green, 2; blue, 27 }  ,draw opacity=1 ]   (118,207.6) -- (478,170.6) ;
\draw [color={rgb, 255:red, 189; green, 16; blue, 224 }  ,draw opacity=1 ]   (138,158) -- (494,208.6) ;
\draw [color={rgb, 255:red, 189; green, 16; blue, 224 }  ,draw opacity=1 ]   (170,90.6) -- (492,218.6) ;
\draw  [fill={rgb, 255:red, 208; green, 2; blue, 27 }  ,fill opacity=1 ] (284,54) .. controls (284,52.34) and (285.34,51) .. (287,51) .. controls (288.66,51) and (290,52.34) .. (290,54) .. controls (290,55.66) and (288.66,57) .. (287,57) .. controls (285.34,57) and (284,55.66) .. (284,54) -- cycle ;
\draw  [fill={rgb, 255:red, 74; green, 144; blue, 226 }  ,fill opacity=1 ] (450,203) .. controls (450,201.34) and (451.34,200) .. (453,200) .. controls (454.66,200) and (456,201.34) .. (456,203) .. controls (456,204.66) and (454.66,206) .. (453,206) .. controls (451.34,206) and (450,204.66) .. (450,203) -- cycle ;
\draw  [color={rgb, 255:red, 0; green, 0; blue, 0 }  ,draw opacity=1 ][fill={rgb, 255:red, 189; green, 16; blue, 224 }  ,fill opacity=1 ] (147,204) .. controls (147,202.34) and (148.34,201) .. (150,201) .. controls (151.66,201) and (153,202.34) .. (153,204) .. controls (153,205.66) and (151.66,207) .. (150,207) .. controls (148.34,207) and (147,205.66) .. (147,204) -- cycle ;
\draw  [fill={rgb, 255:red, 155; green, 155; blue, 155 }  ,fill opacity=1 ] (317,150.1) .. controls (317,148.44) and (318.34,147.1) .. (320,147.1) .. controls (321.66,147.1) and (323,148.44) .. (323,150.1) .. controls (323,151.76) and (321.66,153.1) .. (320,153.1) .. controls (318.34,153.1) and (317,151.76) .. (317,150.1) -- cycle ;
\draw  [fill={rgb, 255:red, 155; green, 155; blue, 155 }  ,fill opacity=1 ] (244,174) .. controls (244,172.34) and (245.34,171) .. (247,171) .. controls (248.66,171) and (250,172.34) .. (250,174) .. controls (250,175.66) and (248.66,177) .. (247,177) .. controls (245.34,177) and (244,175.66) .. (244,174) -- cycle ;
\draw  [fill={rgb, 255:red, 155; green, 155; blue, 155 }  ,fill opacity=1 ] (328,186.1) .. controls (328,184.44) and (329.34,183.1) .. (331,183.1) .. controls (332.66,183.1) and (334,184.44) .. (334,186.1) .. controls (334,187.76) and (332.66,189.1) .. (331,189.1) .. controls (329.34,189.1) and (328,187.76) .. (328,186.1) -- cycle ;

\draw (151,213) node [anchor=north west][inner sep=0.75pt]  [color={rgb, 255:red, 189; green, 16; blue, 224 }  ,opacity=1 ] [align=left] {$P_1$};
\draw (436,206) node [anchor=north west][inner sep=0.75pt]  [color={rgb, 255:red, 74; green, 144; blue, 226 }  ,opacity=1 ] [align=left] {$P_2$};
\draw (295,46) node [anchor=north west][inner sep=0.75pt]  [color={rgb, 255:red, 208; green, 2; blue, 27 }  ,opacity=1 ] [align=left] {$P_3$};
\draw (249,180) node [anchor=north west][inner sep=0.75pt]  [color={rgb, 255:red, 155; green, 155; blue, 155 }  ,opacity=1 ] [align=left] {$P_4$};
\draw (318,128) node [anchor=north west][inner sep=0.75pt]  [color={rgb, 255:red, 155; green, 155; blue, 155 }  ,opacity=1 ] [align=left] {$P_5$};
\draw (329,168) node [anchor=north west][inner sep=0.75pt]  [color={rgb, 255:red, 155; green, 155; blue, 155 }  ,opacity=1 ] [align=left] {$P_6$};

\end{tikzpicture}
\caption{Schematic depiction of the branch locus.}
    \end{center}
\end{figure}

        \noindent	    
	The divisors $ B_{1}, B_{2}, B_{3}, L_{1}, L_{2}, L_{3} $ define a $ \mathbb{Z}^2_2 $-cover $ \xymatrix{\varphi: X \ar[r] & Y_6}  $ branched along $B_{1}+B_{2}+B_{3}$. The surface $ X $ satisfies the following:		
	\begin{align*}
		2K_{X} &\sim \varphi^*\left( 3l-e_1 - e_2 - e_3 - e_4 - e_5 - e_6\right).
	\end{align*}

    \noindent
    Moreover, the surface $ X $ is of general type and minimal and possesses the following invariants:
    $$K_{X}^2= 3, p_g\left( X\right) = 0, q\left( X\right)  = 0, \chi\left( \mathcal{O}_{X}\right) = 1.$$

    \begin{Notation}
    We denote by $E_i := \varphi^*(e_i)_{\text{red}}$ for $i = 1, 2, 3$, 
    $\tilde{E}_i := \varphi^*(e_i)$, for $i = 4, 5, 6$,    
    $H_{145} := \varphi^*(h_{145})_{\text{red}}$,
    $H_{246} := \varphi^*(h_{246})_{\text{red}}$,
    $H_{356} := \varphi^*(h_{356})_{\text{red}}$, and 
    $H_{ij} := \varphi^*(h_{ij})_{\text{red}}$ for all other pairs $1 \le i < j \le 6$.
    \end{Notation}

    \section{Global log canonical thresholds of infinite families of  surfaces of general type as $ \mathbb{Z}_2^n $-covers of $\mathbb{P}^2$ or $\mathbb{P}^1 \times \mathbb{P}^1$.}\label{glctinf}
In this section we give proofs of Theorems \ref{glct_Z_2^n_covering} and \ref{glct_Z_2^n_covering_P^1xP^1}. And we apply the theorems to concrete cases.

    \subsection{Proof of Theorems \ref{glct_Z_2^n_covering} and \ref{glct_Z_2^n_covering_P^1xP^1}}
    
    \subsubsection{Proof of Theorem \ref{glct_Z_2^n_covering}}

Let us proceed to prove Theorem \ref{glct_Z_2^n_covering}. Since $\glct(X,K_X)=2\glct(X, 2K_X)$, it suffices to prove that $\glct(X, 2K_X) = \frac{1}{2(k-6)}$. We divide the proof of this result into the following two propositions.
\begin{Proposition}\label{upper_bound_of_glct}
    $\glct \left( X,2K_{X}\right) \leq \frac{1}{2(k-6)}.$
\end{Proposition}
\begin{proof}
    Consider the divisor:
    \begin{align*}
   		e:=(k-6)l_i \text{ for some }   i \in \{1,2,\ldots, k \}.
   	\end{align*}
    By Proposition \ref{invariants of Z_2^n covers}, $2K_X \sim \varphi^{*}((k -6)l)$. So the divisor 
    $$E:= \varphi^{*}\left( e\right) =2(k-6)\overline{l}_i \sim 2K_{X},$$
    where $2\overline{l}_i = \varphi^{*}(l_i)$. This implies that $ \lct_{P}\left( X,E\right) = \frac{1}{2(k-6)} $ for some $P \in \overline{l}_i$ by \cite[Theorem 6.40]{zbMATH02117209}. Thus we get $\glct \left( X,2K_{X}\right) \leq \frac{1}{2(k-6)}.$
\end{proof}

\begin{Proposition}\label{lower_bound_of_glct}
    $\glct \left( X,2K_{X}\right) \geq \frac{1}{2(k-6)}.$
\end{Proposition}
\begin{proof}
    Suppose, on the contrary, that $\glct(X,2K_X)<\frac{1}{2(k-6)}$. There exists an effective $ \mathbb{Q} $-Cartier divisor $ D \sim_{\mathbb{Q}} 2K_{X}$ such that $ \left( X, \frac{1}{2(k-6)}D\right)  $ is not log canonical at some point $ P \in X $. We denote by $ d:= \varphi\left( D\right)  $. The pair $ \left( \mathbb{P}^2, \frac{1}{2(k-6)}d +\frac{1}{2}(l_1 + l_2 + \cdots + l_k)\right)  $ is not log canonical at $ \varphi\left( P\right)$ by Lemma \ref{log_canonical_between_cover}.\\

    \noindent
    \textbf{Case 1}: Assume that $\varphi\left( P\right) \notin l_1 + l_2 + \cdots + l_k  $. Since the pair $ \left( \mathbb{P}^2, \frac{1}{2(k-6)}d \right)  $ is not log canonical at $ \varphi\left( P\right)$, $ \glct\left( \mathbb{P}^2, d\right) < \frac{1}{2(k-6)} $. We notice that $3d \sim_\mathbb{Q} (k-6)(-K_{\mathbb{P}^2}) $ and that $ \glct\left( \mathbb{P}^2, -K_{\mathbb{P}^2}\right) = \frac{1}{3} $ by \cite[Theorem 1.7]{MR2465686}. We get a contradiction.\\

    \noindent
    \textbf{Case 2}: Assume that $\varphi\left( P\right) \in l_1 + l_2 + \cdots + l_k  $.\\

    \noindent
    \textbf{Case 2.1}:   $\varphi\left( P\right) \in l_i \setminus \bigcup\limits^{k}_{j \neq i} {l_{j}}$ for some $ i \in \{ 1,2,\ldots,k \}$.\\
    \noindent
    Write
    $$d = a_il_i +  \omega,$$
    \noindent
    where $a_i \ge 0$ and $l_i \nsubseteq \textup{supp}(\omega)$. We have that
    $$k-6 = l\cdot d \ge a_i,$$
    \noindent
    where $l$ is a general line in $\mathbb{P}^2$. Because the pair $ \left( \mathbb{P}^2, \frac{1}{2(k-6)}d +\frac{1}{2}l_i\right)  $ is not log canonical at $\varphi(P)$ and $ \frac{a_{i}}{2(k-6)}+\frac{1}{2} \le 1$, the pair $ \left( \mathbb{P}^2, l_{i} +\frac{1}{2(k-6)}\omega\right)  $ is not log canonical at $ \varphi(P) $. By the inversion of adjunction formula \cite[Theorem 7.5]{MR1492525}, we get $ \left( l_{i}, \left.\frac{1}{2(k-6)}\omega\right|_{l_{i}}  \right)  $ is not log canonical at $ \varphi(P) $. 
    This yields that 
    \begin{align*}
		(k-6) -a_i &=\left(  d - a_{i}l_{i} \right)\cdot l_{i} \\ 
        &= \omega\cdot l_{i} \geq \mult_{\varphi(P)}\left( \omega|_{l_{i}} \right) >2(k-6).
    \end{align*}
    \noindent
    This contradicts the inequalities $a_i \ge 0$ and $k \ge 7$.\\

    \noindent
    \textbf{Case 2.2}: $\varphi\left( P\right) \in l_i \cap l_{j}$ for some $ i,j \in \{ 1,2,\ldots,k \}$ and $i \neq j$.\\
    \noindent
    Write
    $$d = a_il_i + a_jl_j +  \omega,$$
    \noindent
    where $a_i, a_j \ge 0$ and $l_i, l_j \nsubseteq \textup{supp}(\omega)$. We notice that
    $$k-6 = l\cdot d \ge a_i +a_j,$$
    \noindent
    where $l$ is a general line in $\mathbb{P}^2$. Because the pair $ \left( \mathbb{P}^2, \frac{1}{2(k-6)}d +\frac{1}{2}l_i +\frac{1}{2}l_j\right)  $ is not log canonical at $\varphi(P)$ and $ \frac{a_{i}}{2(k-6)}+\frac{1}{2} \le 1$, the pair $ \left( \mathbb{P}^2, l_{i} +\frac{1}{2(k-6)}\omega + \left(\frac{a_{j}}{2(k-6)}+\frac{1}{2}\right)l_j\right)  $ is not log canonical at $ \varphi(P) $. By the inversion of adjunction formula, we get $ \left( l_{i}, \left.\left(\frac{1}{2(k-6)}\omega+ \left(\frac{a_{j}}{2(k-6)}+\frac{1}{2}\right)l_j\right)\right|_{l_{i}}  \right)  $ is not log canonical at $ \varphi(P) $. 
    This yields that 
    \begin{align*}
		(k-6) -a_i -a_j +a_j + (k-6)&=\left(  d - a_{i}l_{i}- a_{j}l_{j} \right)\cdot l_{i} + (a_j + (k-6))l_j\cdot l_{i}\\ 
        &= (\omega + (a_j + (k-6))l_j)\cdot l_{i}\\
        &\geq \mult_{\varphi(P)}\left( (\omega + (a_j + (k-6))l_j)|_{l_{i}} \right) >2(k-6).
    \end{align*}
    \noindent
    So we obtain $a_i < 0$. This contradicts the inequality $a_i \ge 0$.    
    \end{proof}

    \subsubsection{Proof of Theorem \ref{glct_Z_2^n_covering_P^1xP^1}}

    We now proceed to prove Theorem \ref{glct_Z_2^n_covering_P^1xP^1}. Since $\glct(X,K_X)=2\glct(X, 2K_X)$, it suffices to prove that $\glct(X, 2K_X) = \frac{1}{2(k'-4)}$. The proof is structured around the following two propositions.

  \begin{Lemma}\label{glct of the sum of two fibers}
For any integers $a,b$ with $a\ge b>0$, $$\glct(\mathbb{P}^1\times\mathbb{P}^1, aF+bG)=\frac{1}{a}.$$
\end{Lemma}
\begin{proof}
By Lemma \ref{comarelct} we have $$\glct(\mathbb{P}^1\times\mathbb{P}^1, aF+bG)\le\frac{1}{a}.$$

Now it suffices to show that $\glct(\mathbb{P}^1\times\mathbb{P}^1, aF+bG)\ge\frac{1}{a}$. Suppose that $\glct(\mathbb{P}^1\times\mathbb{P}^1, aF+bG)<\frac{1}{a}$. Then there exists a $\mathbb{Q}$-effective divisor $C$ such that $C\sim_{\mathbb{Q}} aF+bG$ and the pair $(\mathbb{P}^1\times\mathbb{P}^1, \frac{1}{a}C)$ is not log canonical at some point $Q\in \mathbb{P}^1\times\mathbb{P}^1$.

Write $$C=a_1 F + b_1 G+ \omega,$$ where $a_1, b_1 \ge 0$ and $F,G\nsubseteq \textup{supp}(\omega)$.
Since $a=G(aF+bG)=GC\ge a_1$ we have $1\ge \frac{a_1}{a}$. Thus the pair $(\mathbb{P}^1\times\mathbb{P}^1, F+\frac{b_1}{a}G+\frac{1}{a}\omega)$ is not log canonical at $Q$. By the inversion of adjunction formula the pair $(F,\left( \frac{b_1}{a}G+\frac{1}{a}\omega\right)\mid_ F)$ is not log canonical at $Q$. Then $$b=(C-a_1F)F=(b_1G+\omega)F\ge\mult_{Q}(b_1G+\omega)\mid_F)>a$$
which is a contradiction.
\end{proof}

\begin{Proposition}\label{upper_bound_of_glct_covering_P^1xP^1}
    $\glct \left( X,2K_{X}\right) \leq \frac{1}{2(k'-4)}.$
    \end{Proposition}
    \begin{proof}
   Consider the following divisor:
    \begin{align*}
   		e:=(k-4)F_i + (k'-4)G_j\hskip 0,5cm \text{for some } i \in \{1,2,\ldots, k \} \text{ and }  j \in \{1,2,\ldots, k' \}.
   	\end{align*}
    Since $2K_X = \varphi^{*}((k-4)F + (k'-4)G)$ the divisor 
    $$E:= \varphi^{*}\left( e\right) =2(k-4)\overline{F}_i + 2(k'-4)\overline{G}_j \sim 2K_{X},$$
    where $2\overline{F}_i = \varphi^{*}(F_i)$ and $2\overline{G}_j = \varphi^{*}(G_j)$. This implies that $ \lct_{P}\left( X,E\right) = \frac{1}{2(k'-4)} $ at $P \in \overline{F}_i \cap \overline{G}_j$ by Lemma \ref{comarelct}. Thus $\glct \left( X,2K_{X}\right) \leq \frac{1}{2(k'-4)}.$
\end{proof}

    \begin{Proposition}\label{lower_bound_of_glct_covering_P^1xP^1}
    $\glct \left( X,2K_{X}\right) \geq \frac{1}{2(k'-4)}.$
    \end{Proposition}
    \begin{proof}
        Suppose, on the contrary, that $\glct(X,2K_X)<\frac{1}{2(k'-4)}$. There exists an effective $ \mathbb{Q} $-Cartier divisor $ D \sim_{\mathbb{Q}} 2K_{X}$ such that $ \left( X, \frac{1}{2(k'-4)}D\right)  $ is not log canonical at some point $ P \in X $. We denote by $ d:= \varphi\left( D\right)  $. Then the pair $$ \left( \mathbb{P}^1 \times \mathbb{P}^1, \frac{1}{2(k'-4)}d +\frac{1}{2}\left(\sum \limits_{i=1}^{k}{F_i} + \sum \limits_{j=1}^{k'}{G_j}\right)\right)  $$ is not log canonical at $ \varphi\left( P\right)$.\\

    \noindent
    \textbf{Case 1}: Assume that $\varphi\left( P\right) \notin \textup{supp}\left(\sum \limits_{i=1}^{k}{F_i} + \sum \limits_{j=1}^{k'}{G_j}\right)  $. Since $ \left( \mathbb{P}^1 \times \mathbb{P}^1, \frac{1}{2(k'-4)}d \right)  $ is not log canonical at $ \varphi\left( P\right)$, $ \glct\left( \mathbb{P}^1 \times \mathbb{P}^1, d\right) < \frac{1}{2(k'-4)} $. On the other hand since $d\sim_{\mathbb{Q}}(k-4)F+(k'-4)G$ we obtain $\glct(\mathbb{P}^1\times\mathbb{P}^1,d)=\frac{1}{k'-4}$ by Lemma \ref{glct of the sum of two fibers}.  We get a contradiction.\\

    \noindent
    \textbf{Case 2}: Assume that $\varphi\left( P\right) \in \textup{supp}\left(\sum \limits_{i=1}^{k}{F_i} + \sum \limits_{j=1}^{k'}{G_j}\right)  $.\\

    \noindent
    \textbf{Case 2.1}:   $\varphi\left( P\right) \in \left(F_i \setminus \bigcup\limits^{k'}_{j =1} {G_{j}}\right) \bigcup \left(G_{i'} \setminus \bigcup\limits^{k}_{j'=1 } {F_{j'}}\right)$ for some $i \in \{ 1,2,\ldots,k \}$ and $i' \in \{ 1,2,\ldots,k' \}$.\\
    \noindent
   (1)  Assume that $\varphi\left( P\right) \in F_i \setminus \bigcup\limits^{k'}_{j =1} {G_{j}} $. Write
    $$d = a_iF_i +  \omega,$$
    \noindent
    where $a_i \ge 0$ and $F_i \nsubseteq \textup{supp}(\omega)$. We have that
    $$k-4 = G\cdot d \ge a_i,$$
    \noindent
    where $G \in | G |$ is a general fiber in $\mathbb{P}^1 \times \mathbb{P}^1$. Because the pair $ \left( \mathbb{P}^1 \times \mathbb{P}^1, \frac{1}{2(k'-4)}d +\frac{1}{2}F_i\right)  $ is not log canonical at $\varphi(P)$ and $ \frac{a_{i}}{2(k'-4)}+\frac{1}{2} \le 1$, the pair $ \left( \mathbb{P}^1 \times \mathbb{P}^1, F_{i} +\frac{1}{2(k'-4)}\omega\right)  $ is not log canonical at $ \varphi(P) $. By the inversion of adjunction formula, we get $ \left( F_{i}, \frac{1}{2(k'-4)}\omega\mid_{F_{i}}  \right)  $ is not log canonical at $ \varphi(P) $. 
    This yields that 
    \begin{align*}
		k'-4  &=\left(  d - a_{i}F_{i} \right)\cdot F_{i} \\ 
        &= \omega\cdot F_{i} \geq \mult_{\varphi(P)}\left( \omega\mid_{F_{i}} \right) >2(k'-4).
    \end{align*}
    \noindent
    This contradicts the inequality $k' \ge 5$. \\
    
    \noindent
    (2) Assume that  $\varphi(P)\in G_{i'} \setminus \bigcup\limits^{k}_{j'=1 } {F_{j'}}$. Write
    $$d = b_{i'}G_{i'} +  \omega,$$
    \noindent
    where $b_{i'} \ge 0$ and $G_{i'} \nsubseteq \textup{supp}(\omega)$. We have that
    $$k'-4 = F\cdot d \ge b_{i'},$$
    \noindent
    where $F \in | F |$ is a general fiber in $\mathbb{P}^1 \times \mathbb{P}^1$. Because the pair $ \left( \mathbb{P}^1 \times \mathbb{P}^1, \frac{1}{2(k'-4)}d +\frac{1}{2}G_{i'}\right)  $ is not log canonical at $\varphi(P)$ and $ \frac{b_{i'}}{2(k'-4)}+\frac{1}{2} \le 1$, the pair $ \left( \mathbb{P}^1 \times \mathbb{P}^1, G_{i'} +\frac{1}{2(k'-4)}\omega\right)  $ is not log canonical at $ \varphi(P) $. By the inversion of adjunction formula, we get $ \left( G_{i'}, \frac{1}{2(k'-4)}\omega\mid_{G_{i'}}  \right)  $ is not log canonical at $ \varphi(P) $. 
    This yields that 
    \begin{align*}
		k-4  &=\left(  d - b_{i'}G_{i'} \right)\cdot G_{i'} \\ 
        &= \omega\cdot G_{i'} \geq \mult_{\varphi(P)}\left( \omega\mid_{G_{i'}} \right) >2(k'-4).
    \end{align*}
    \noindent
    This contradicts the inequality $k' \ge k\ge 5$. \\
    
    \noindent
    \textbf{Case 2.2}: $\varphi\left( P\right) \in F_i \cap G_{j}$ for some $i \in \{ 1,2,\ldots,k \}$ and $j \in \{ 1,2,\ldots,k' \}$.\\
    \noindent
    Write
    $$d = a_iF_i + b_jG_j +  \omega,$$
    \noindent
    where $a_i, b_j \ge 0$ and $F_i, G_j \nsubseteq \textup{supp}(\omega)$. We notice that
    $$k-4 = G\cdot d \ge a_i,$$
    \noindent
    where $G \in | G |$ is a general fiber in $\mathbb{P}^1 \times \mathbb{P}^1$. Because the pair $ \left( \mathbb{P}^1 \times \mathbb{P}^1, \frac{1}{2(k'-4)}d +\frac{1}{2}F_i +\frac{1}{2}G_j\right)  $ is not log canonical at $\varphi(P)$ and $ \frac{a_{i}}{2(k'-4)}+\frac{1}{2} \le 1$, the pair $$ \left( \mathbb{P}^1 \times \mathbb{P}^1, F_{i} +\frac{1}{2(k'-4)}\omega + \left(\frac{b_{j}}{2(k'-4)}+\frac{1}{2}\right)G_j\right)  $$ is not log canonical at $ \varphi(P) $. By the inversion of adjunction formula, we get $ \left( F_{i}, \left(\frac{1}{2(k'-4)}\omega+ \left(\frac{b_{j}}{2(k'-4)}+\frac{1}{2}\right)G_j\right)\mid_{F_{i}}  \right)  $ is not log canonical at $ \varphi(P) $. 
    This yields that 
    \begin{align*}
		(k'-4) -b_j +b_j + (k'-4)&=\left(  d - a_{i}F_{i}- b_{j}G_{j} \right)\cdot F_{i} + (b_j + (k'-4))G_j\cdot F_{i}\\ 
        &= (\omega + (b_j + (k'-4))G_j)\cdot F_{i}\\
        &\geq \mult_{\varphi(P)}\left( (\omega + (b_j + (k'-4))G_j)\mid_{F_{i}} \right) >2(k'-4).
    \end{align*}
    \noindent
    So we obtain a contradiction.
     \end{proof}

    \subsection{Applications of Theorems \ref{glct_Z_2^n_covering} and \ref{glct_Z_2^n_covering_P^1xP^1}} \label{application_glct_Z_2^n}

    Let us highlight some consequences of Theorem \ref{glct_Z_2^n_covering}. By choosing each branch component $D_{\sigma}$ to be a line, it can be shown that there exists an infinite sequence of smooth surfaces of general type $X_n$ for $n \geq 3$ such that 
    $$
    \mathrm{glct}(X_n) = \frac{1}{2^n - 7}.
    $$
    \begin{Corollary}\label{Campedelli's generation}
    Let $n$ be an integer such that $n \ge 3$, and $l_{\sigma} $, where $\sigma \in \mathbb{Z}^n_2 \setminus \{0 \}$, denote $2^{n}-1$ lines in general position in $\mathbb{P}^2$. Suppose that $ \xymatrix{\varphi: X_n \ar[r] & \mathbb{P}^2}  $ is a smooth $ \mathbb{Z}^n_2 $-cover, branched along the total branch divisor $ \sum\limits_{\sigma \ne 0}{D_{\sigma}} $, where $D_{\sigma} = l_{\sigma}$, for all $ \sigma \ne 0$. Then $X_n$ is a smooth surface of general type such that
    \begin{align*}
        K_{X_n}^2 = 2^{n-2}(2^n -7)^2, \hskip 0.2cm p_g(X_n) = (2^{n}-1)(1 + 2^{n-3}(2^{n-2}-3)), \hskip 0.2cm q(X_n) = 0.
    \end{align*}
    Furthermore, 
    $$\glct{(X_n)} = \frac{1}{2^n -7}.$$
    \end{Corollary}
    \noindent
    We observe that the surface $X_3$ for $n=3$ described in Corollary \ref{Campedelli's generation} is a classical Campedelli surface (cf. \cite{MR5070691}). And the Burniat surface with $K^2 = 2$ can be seen as a classical Campedelli surface (with fundamental group $\mathbb{Z}_2^3$) (cf. \cite[Section 5]{zbMATH05885266}).

    \begin{Corollary}\label{glct_Campedelli_surfaces}
        Let $X$ be a Burniat surface with $K_X^2 = 2$. Then
         $$\glct{(X)} = 1.$$
    \end{Corollary}

    An alternative selection of the branch components, as described in Corollary \ref{Persson's generation}, demonstrates the existence of an infinite sequence of smooth surfaces of general type $X_n$ for $n \ge 4$ such that 
    $$\glct{(X_n)} = \frac{1}{2^{n-1} -6}.$$

    \begin{Corollary}\label{Persson's generation}
    Let $n$ be an integer such that $n \ge 4$, and $\sigma = (1,*,*,\ldots,*) $ be an element in $ \mathbb{Z}^n_2$, where the first coordinate is fixed at $1$ and the remaining coordinates are arbitrary. Let $l_{\sigma}$, where $\sigma =(1,*,*,\ldots,*) \in \mathbb{Z}^n_2$, denote $2^{n-1}$ lines in general position in $\mathbb{P}^2$. Suppose that $ \xymatrix{\varphi: X_n \ar[r] & \mathbb{P}^2}  $ is a smooth $ \mathbb{Z}^n_2 $-cover, branched along the total branch divisor $ \sum\limits_{\sigma \ne 0}{D_{\sigma}} $, where $D_{\sigma} = l_{\sigma}$,  for each $\sigma=(1,*,*,\ldots,*) $ and $D_{\sigma} = 0$ otherwise. Then $X_n$ is a smooth surface of general type such that
    \begin{align*}
        K_{X_n}^2 = 2^{n-2}(2^{n-1} -6)^2, \hskip 0.2cm p_g(X_n) = 2^{3n-7} -11\cdot 2^{2n-6} + 2^{n} -1, \hskip 0.2cm q(X_n) = 0.
    \end{align*}
    Furthermore, 
    $$\glct{(X_n)} = \frac{1}{2^{n-1} -6}.$$
    \end{Corollary}
    \noindent
    We emphasize that when $n =4$, the surface $X_4$ described in Corollary \ref{Persson's generation} is a Persson surface. Persson surface is a double cover of a Campedelli surface, branched along a bicanonical divisor \cite[Example 5.8]{MR527234}. It represents the first example of a surface of general type where the canonical map has a degree greater than $9$. Furthermore, Persson surface can be constructed as a $\mathbb{Z}^4_2$-cover of $\mathbb{P}^{2}$ \cite[Example 4.6]{MR4592547}.\\

    Similar to the case of Theorem \ref{glct_Z_2^n_covering}, there are infinitely many sequences of surfaces of general type which can apply Theorem \ref{glct_Z_2^n_covering_P^1xP^1}. Here, we highlight some distinct known examples that utilize the same total branch locus $F_1 + F_2 + \cdots + F_6 + G_1 + G_2 + \cdots + G_6$ as in Theorem \ref{glct_Z_2^n_covering_P^1xP^1}.

    \begin{Example}\label{Example 1}
    Let $F_1, F_2, \ldots, F_6 \in | F |$, $G_1, G_2, \ldots, G_6 \in | G |$ be distinct divisors in $\mathbb{P}^1 \times \mathbb{P}^1$. Consider $D_{10} = F_1 + F_2 + \cdots + F_6$, $D_{01} = G_1 + G_2 + \cdots + G_6$, and $D_{11} = 0$. Let $ \xymatrix{\varphi: X \ar[r] & \mathbb{P}^1 \times \mathbb{P}^1}  $ be a $ \mathbb{Z}^2_2 $-cover with the building data $\{ D_{10}, D_{01}, D_{11}\}$. Then
    $$K_X^2 = 8, p_g(X) = q(X) = 4, \glct{(X)} = \frac{1}{2}.$$
    \end{Example}

\noindent
For surfaces with $p_g = q = 4$, we refer to Beauville's work in the appendix of
\cite{MR688038}, or to \cite{MR1895196}.

An irregular surface of general type with canonical map of degree $16$ was constructed using a $ \mathbb{Z}^4_2 $-cover of $\mathbb{P}^1 \times \mathbb{P}^1$, ramified along the $F_1 + F_2 + \cdots + F_6 + G_1 + G_2 + \cdots + G_6$ by the first author \cite{MR4008073}. 
\begin{Example} \label{Example 2}
    Let $F_1, F_2, \ldots, F_6 \in | F |$, $G_1, G_2, \ldots, G_6 \in | G |$ be distinct divisors in $\mathbb{P}^1 \times \mathbb{P}^1$. Consider $D_{0100} = F_1 + F_2$, $D_{0101} = F_3 + F_4$, $D_{1000} = F_5 + F_6$, $D_{0110} = G_1 $, $D_{0111} = G_2$, $D_{1001} = G_3 + G_4$, $D_{1010} = G_5$, $D_{1011} = G_6$ and $D_{\sigma} = 0$ for the rest of $D_{\sigma}$. Let $ \xymatrix{\varphi: X \ar[r] & \mathbb{P}^1 \times \mathbb{P}^1}  $ be a $ \mathbb{Z}^4_2 $-cover with the building data $\{ D_{\sigma}\}_{\sigma \in \mathbb{Z}^4_2\setminus\{(0,0,0,0)\}}$. Then
    $$K_X^2 = 32, p_g(X) =  4, q(X) = 1, \glct{(X)} = \frac{1}{2}.$$
\end{Example}

In \cite{MR4621715}, Gleissner, Pignatelli, and Rito constructed surfaces of general type with canonical map of degree $32$ using $ \mathbb{Z}^4_2 $-covers of $\mathbb{P}^1 \times \mathbb{P}^1$, ramified along the $F_1 + F_2 + \cdots + F_6 + G_1 + G_2 + \cdots + G_6$.
\begin{Example} \label{Example 3}
    Let $F_1, F_2, \ldots, F_6 \in | F |$, $G_1, G_2, \ldots, G_6 \in | G |$ be distinct divisors in $\mathbb{P}^1 \times \mathbb{P}^1$. Consider $D_{1000} = F_1 $, $D_{0100} = F_2$, $D_{0010} = F_3$, $D_{0001} = F_4$, $D_{1010} = F_5$, $D_{0101} = F_6$, $D_{0111} = G_1 $, $D_{1011} = G_2$, $D_{1101} = G_3$, $D_{1110} = G_4$, $D_{1001} = G_5$, $D_{0110} = G_6$, and $D_{\sigma} = 0$ for the rest of $D_{\sigma}$. Let $ \xymatrix{\varphi: X \ar[r] & \mathbb{P}^1 \times \mathbb{P}^1}  $ be a $ \mathbb{Z}^4_2 $-cover with the building data $\{ D_{\sigma}\}_{\sigma \in \mathbb{Z}^4_2\setminus\{(0,0,0,0)\}}$. Then
    $$K_X^2 = 32, p_g(X) =  3, q(X) = 0, \glct{(X)} = \frac{1}{2}.$$
\end{Example}

    \section{Global log canonical threshold of Burniat surfaces}\label{glctbur}
    This section is devoted to prove Theorem \ref{GLCT_Burniatsurfaces_K^2=4}.

    \subsection{Global log canonical threshold of the secondary Burniat surfaces with $ K^2 = 4 $}
     
    \subsubsection{$K^2 = 4 $ non-nodal type}
    The aim of this section is to prove that $ \glct \left( X,K_{X}\right) = \frac{2}{3} $, where $X$ is a Burniat surface with $K_X^2=4$ of non-nodal type (see Section \ref{Construction_of_Burniat_surface_with_K^2=4_non_nodal_type}). Since $\glct(X,K_X)=2\glct(X, 2K_X)$, it suffices to prove that $\glct(X, 2K_X) = \frac{1}{3}$. We divide the proof of this result into the following two propositions.
    \begin{Proposition}\label{upperboundofglct_surfaces_with_K2=4_nonnodaltype} Let $X$ be a Burniat surface with $K_X^2=4$ of non-nodal type. Then $ \glct \left( X,2K_{X}\right) \le \frac{1}{3} $.
   	\end{Proposition}
    
   \begin{proof}
   Let $ P\in X $ be such that $ \varphi\left( P\right) \in  e_3\cap h_{35}$. Consider the following divisor:
   	\begin{align*}
   		e:=3h_{35} + e_3 + e_5 + h_{12}+ h_{14} + h_{24}\sim 6l-2e_1 - 2e_2 - 2e_3 - 2e_4 - 2e_5.
   	\end{align*}
   \noindent
   The divisor $E:= \varphi^{*}\left( e\right) =6H_{35} + 2E_3 + 2E_5 + 2H_{12}+ 2H_{14} + 2H_{24} \sim 4K_{X}$. This implies that $ \lct\left( X,E\right) = \frac{1}{6} $ at $P$ by Lemma \ref{comarelct}. So we get $$ \glct\left( X,2K\right) = 2\glct\left( X,4K\right) \le 2.\frac{1}{6} = \frac{1}{3}. $$
   \end{proof}

   \begin{Proposition}\label{lowerboundofglct} Let $X$ be a Burniat surface with $K_X^2=4$ of non-nodal type. Then 
   	$ \glct \left( X,2K_{X}\right) \ge \frac{1}{3} $.
    \begin{proof} 
     Suppose on the contrary that $\glct(X,2K_X)<\frac{1}{3}$. There exists an effective $ \mathbb{Q} $-Cartier divisor $ D \sim_{\mathbb{Q}} 2K_{X}$ such that $ \left( X, \frac{1}{3}D\right)  $ is not log canonical at some point $ P \in X $. By \cite[Proposition 9.5.11]{MR2095471}, we have that $\mult_P\left( D\right) >3.$

    \noindent
    We denote by $ d:= \varphi\left( D\right)  $. By Lemma \ref{log_canonical_between_cover}, $ \left( Y_5, \frac{1}{3}d +\frac{1}{2}B\right)  $ is not log canonical at $ \varphi\left( P\right)$ where $B:=B_1+B_2+B_3$.\\ 
	
    \noindent
    \textbf{Case 1}: Assume that $\varphi\left( P\right) \notin B  $. Since $\left( Y_5,  \frac{1}{3}d \right)$ is not log canonical at $ \varphi\left( P\right)$, $ \glct\left( Y_5, d\right) < \frac{1}{3} $. Recall that $Y_5 $ is a del Pezzo surface blow up at $5$ general points of $\mathbb{P}^2$. We notice that $d \sim_\mathbb{Q} -K_{Y_5} $ and that $ \glct\left( Y_5, -K_{Y_5}\right) = \frac{2}{3} $ by \cite[Theorem 1.7]{MR2465686}. We get a contradiction.\\

\noindent
\textbf{Case 2}: Assume that $\varphi\left( P\right) \in B  $.\\
\noindent
    \textbf{Case 2.1}:   $\varphi\left( P\right) \in e_i \setminus \bigcup\limits^{5}_{\substack{j=1 \\ j \neq i}} {h_{ij}}$ for some $i \in \{ 1,2,3 \}$.\\
    
    \noindent For simplicity, we show that we get a contradiction if $i = 1$. The same proof works for other $i$. Suppose that $\varphi\left( P\right) \in e_{1} \setminus (h_{12} \cup h_{13} \cup h_{14} \cup h_{15})$.\\  

\noindent
    Write
    $$d = a_1e_1 +  \omega,$$
    \noindent
    where $a_1 \ge 0$ and $e_1 \nsubseteq \textup{supp}(\omega)$. Since $\glct(Y_5,d) = \frac{2}{3}$, we get $a_1 \le \frac{3}{2}$.\\

    \noindent
    Because $(Y_5,\frac{1}{3}d+\frac{1}{2}e_1)$ is not log canonical at $\varphi(P)$ and $ \frac{a_{1}}{3}+\frac{1}{2} \le 1$, the pair $ \left( Y_5, e_{1} +\frac{1}{3}\omega\right)  $ is not log canonical at $ \varphi(P) $. By the inversion of adjunction formula, we get $ \left( e_{1}, \left.\frac{1}{3}\omega\right|_{e_{1}}  \right)  $ is not log canonical at $ \varphi(P) $. 
    This yields that 
    \begin{align*}
		1+ a_{1}  &=\left(  d - a_{1}e_{1} \right)\cdot e_{1} \\ 
        &= \omega\cdot e_{1} \geq \mult_{\varphi(P)}\left( \omega|_{e_{1}} \right) >3.
    \end{align*}

    \noindent
    This contradicts the inequality $a_1 \le \frac{3}{2}$.\\

    \noindent
    \textbf{Case 2.2}:   $\varphi\left( P\right) \in h_{ij} \setminus \left(e_i \cup e_j \cup \bigcup\limits_{\substack{l,k \in \{ 1,2,3,4,5 \} \setminus \{ i,j \} \\ (l,k)\neq(4,5)}} {h_{lk}}\right)$ for some $i \in \{ 1,2,3 \}$ and $j \in \{ 1,2,3,4,5 \} \setminus \{ i \}$.\\

    \noindent
    Similarly to Case 2.1.\\

    \noindent
    \textbf{Case 2.3}: $\varphi\left( P\right) \in e_k \cap h_{kl}$ for some $k \in \{ 1,2,3 \}$, $l \in \{ 1,2,3,4,5 \} \setminus \{ k \}$.\\
    \noindent For simplicity, we show that we get a contradiction if $(k,l) = (1,2)$. The same proof works for other $(k,l)$.\\

    \noindent
    We have $(Y_5,\frac{1}{3}d+\frac{1}{2}e_1+\frac{1}{2}h_{12})$ is not log canonical at $\varphi(P)$ where $d=\varphi(D)\sim_\mathbb{Q} -K_Y$.
    Write
    $$d = a_1e_1 + a_{12}h_{12}+ a_{13}h_{13} + a_{14}h_{14}+ a_{15}h_{15}+ \omega,$$
    \noindent
    where $a_1, a_{12}, a_{13}, a_{14}, a_{15} \ge 0$ and $e_1, h_{12}, h_{13}, h_{14}, h_{15} \nsubseteq \textup{supp}(\omega)$.\\
   
    \noindent
    \textbf{Claim}: $a_{12} \le 1.$\\

    \noindent
    Remark that $ \left( X, \frac{1}{3}\left( 3H_{12} + E_1 + E_2 + H_{34} + H_{35} + \frac{1}{2}\tilde{H}_{45} \right) \right)  $ is log canonical at $P $. By \cite[Lemma 2.2]{MR4216579}, we can assume that $ \support\left( D\right)  $ does not contain at least one component of $ 3H_{12} + E_1 + E_2 + H_{34} + H_{35} + \frac{1}{2}\tilde{H}_{45}$ and $(X,\frac{1}{3}D)$ is not log canonical at $P$.  Note that $\support\left( D\right)$  contains both $E_1$ and $H_{12}$, since otherwise it leads to a contradiction   \begin{align*}2&=E_1\cdot D\geq \mult_P(E_1)\cdot \mult_P(D)=\mult_P(D)>3 \\
    (\textup{or } 2&=H_{12}\cdot D\geq \mult_P(H_{12})\cdot \mult_P(D)=\mult_P(D)>3).
    \end{align*} 
    Therefore, 
    $$ H_{34} \not\subseteq \support\left( D\right),\  H_{35} \not\subseteq \support\left( D\right),\ \tilde{H}_{45} \not\subseteq \support\left( D\right)\ \textrm{or}\   E_{2} \not\subseteq \support\left( D\right).$$
    In particular, $$ h_{34} \not\subseteq \support\left( d\right),\  h_{35} \not\subseteq \support\left( d\right),\ h_{45} \not\subseteq \support\left( d\right)\ \textrm{or}\   e_{2} \not\subseteq \support\left( d\right).$$
    \noindent
    If $h_{ij} \not\subseteq \support\left( d\right)$ for $(i,j)=(3,4), (3, 5)$ or $(4,5)$, then 
    $$1 =h_{ij}\cdot d = h_{ij}\cdot \left( a_1e_1 + a_{12}h_{12}+ a_{13}h_{13} + a_{14}h_{14}+ a_{15}h_{15} + \omega\right) \ge a_{12}.$$




    \noindent
    If $e_{2} \not\subseteq \support\left( d\right)$, then 
    $$1 =e_{2}\cdot d = e_{2}\cdot \left( a_1e_1 + a_{12}h_{12}+ a_{13}h_{13} + a_{14}h_{14}+ a_{15}h_{15} + \omega\right) \ge a_{12}.$$

    \noindent
    Therefore, $a_{12} \le 1.$\\

    \noindent
    \textbf{Claim}: $a_{1} \leq 1 + a_{1k}$ for all $k=3,4$ and $5$.\\

    \noindent 
    \noindent 
    For all $k=3,4,5$, we have that  
    \begin{align*}
    0\leq h_{1k}\cdot \omega &=h_{1k}\cdot (d-a_1e_1 - a_{12}h_{12}- a_{13}h_{13} - a_{14}h_{14}- a_{15}h_{15}-a_{35}h_{35})\\
    &=1-a_1+a_{1k}.
    \end{align*}
    \noindent 
    Thus $a_1 \leq 1 + a_{1k}$.\\
    
    \noindent  Denote by $m = \mult_{\varphi(P)}(\omega)$.\\

    \noindent
    \textbf{Claim}: $a_1+m\leq 2$.\\
    
    \noindent
    Since $\left( Y_5, \frac{1}{3}d+\frac{1}{2}e_1+\frac{1}{2}h_{12}\right)  $ is not log canonical at $ \varphi(P) $ and $ \frac{a_{12}}{3}+\frac{1}{2} < 1$, the pair $ \left( Y_5, h_{12} +\frac{1}{3}\left(\left(a_{1}+\frac{3}{2}\right)e_{1}+\omega\right)\right)  $ is not log canonical at $ \varphi(P) $. By the inversion of adjunction formula, we get $ \left( h_{12}, \frac{1}{3}\left.\left(\left(a_{1}+\frac{3}{2}\right)e_{1}+\omega\right)\right|_{h_{12}}  \right)  $ is not log canonical at $ \varphi(P) $. 
    This yields that 
    \begin{align*}
		\frac{5}{2} + a_{12}  &=\frac{3}{2}+\left(  d - a_{12}h_{12}- a_{13}h_{13} - a_{14}h_{14}- a_{15}h_{15} \right)\cdot h_{12} \\ 
        &= \frac{3}{2}+(a_{1}e_{1}+\omega)\cdot h_{12} \geq \mult_{\varphi(P)}\left(\left. \left(\left(a_{1}+\frac{3}{2}\right)e_{1}+\omega\right)\right|_{h_{12}} \right)\\
        &=a_1+\frac{3}{2}+m.
    \end{align*}

    \noindent
   Since $a_{12}\le 1$ we have $a_1+m\leq 2$.\\




    \noindent
    Let $\pi: Y \longrightarrow Y_5$ be the blowing-up of $Y_5$ at $\varphi(P)$. We denote by $F$ the exceptional divisor, by $e_1^1$, $h_{12}^1$, $\omega_1$, $d_1$ the proper transform of $e_1$, $h_{12}$, $\omega$, $d$, respectively, and by $h^1_{13}$, $h^1_{14}$, $h^1_{15}$ the pull-back of $h_{13}$, $h_{14}$, $h_{15}$, respectively. So we have that
    \begin{align*}
        d_1 = a_1e^1_1 + a_{12}h^1_{12}+ a_{13}h^1_{13} + a_{14}h^1_{14}+ a_{15}h^1_{15} + \omega_1
    \end{align*}

    \noindent
    and
    $$d_1 = \pi^{*}(d) - (a_1 + a_{12} + m)F,\ e^1_1 = \pi^{*}(e_1) - F,\ h^1_{12} = \pi^{*}(h_{12}) - F,\ \omega_1=\pi^{*}(\omega)-mF.$$
    \noindent
    By \cite[Remark 2.6]{MR2465686}, $\left( Y, \frac{1}{3}d_1 +\frac{1}{2}e_1^1+\frac{1}{2}h_{12}^1+\frac{a_1 + a_{12} + m}{3}F\right)  $ is not log canonical at some point $Q \in F$.\\

    \noindent
    Suppose that $Q \in F\setminus (e_1^1 \cup h_{12}^1)$.\\

    \noindent
    So $\left( Y, \frac{1}{3}\omega_1 +\frac{a_1 + a_{12} + m}{3}F\right)  $ is not log canonical at $ Q $. Since $$\frac{a_1 + a_{12} + m}{3} \le 1,$$
    \noindent
    the pair $\left( Y, \frac{1}{3}\omega_1 +F\right)  $ is not log canonical at $ Q $. By the inversion of adjunction formula, we get $ \left( F, \left.\frac{1}{3}\omega_1\right|_{F}  \right)  $ is not log canonical at $ Q $. 
    This implies that
    \begin{align*}
        3 < F\cdot  \omega_1 &= F\cdot (d_1 - a_1e^1_1 - a_{12}h^1_{12} - a_{13}h^1_{13} - a_{14}h^1_{14}- a_{15}h^1_{15})\\
        &=(a_1 + a_{12} + m) - a_1 -a_{12} = m.
    \end{align*}

    \noindent
    Thus we get $m >3$ which contradicts to the above inequalities $a_1 + m \leq 2$ and $a_1\ge 0$.\\

    \noindent
    Suppose that $Q \in F \cap h_{12}^1$.\\

    \noindent
    So $\left( Y, \left(\frac{a_{12}}{3}+\frac{1}{2}\right)h^1_{12}+\frac{1}{3}\omega_1 +\frac{a_1 + a_{12} + m}{3}F\right)  $ is not log canonical at $ Q $. Since $\frac{a_{12} }{3}+\frac{1}{2} < 1,$ the pair $\left( Y, h^1_{12}+\frac{1}{3}\omega_1 +\frac{a_1 + a_{12} + m}{3}F\right)  $ is not log canonical at $ Q $. By the inversion of adjunction formula, we get $ \left( h^1_{12}, \left.\left(\frac{1}{3}\omega_1+\frac{a_1 + a_{12} + m}{3}F\right)\right|_{h^1_{12}}  \right)  $ is not log canonical at $ Q $. 
    This implies that
    \begin{align*}
        3 &<h^1_{12} \cdot ( \omega_1 + (a_1 + a_{12} +m)F)\\ 
        &= h^1_{12}\cdot (d_1 - a_1e^1_1 - a_{12}h^1_{12} - a_{13}h^1_{13} - a_{14}h^1_{14}- a_{15}h^1_{15} ) + (a_1 + a_{12} +m)\\
        &=(1 - (a_1 + a_{12} + m)) + 2a_{12} + (a_1 + a_{12} +m).
    \end{align*}

    \noindent
    Thus we get $1  < a_{12}$ which contradicts to the inequality $a_{12} \leq 1$.\\

    \noindent
    Suppose that $Q \in F \cap e^1_{1}$.\\

    \noindent
    So $\left( Y, \left(\frac{1}{2}+\frac{a_{1}}{3}\right)e^1_{1}+\frac{1}{3}\omega_1 +\frac{a_1 + a_{12} + m}{3}F\right)  $ is not log canonical at $ Q $. Note $a_1\le\frac{3}{2}$ because of $\glct(Y_4,d)=\glct(Y_4,-K_{Y_4})=\frac{2}{3}$  by \cite[Theorem 1.7]{MR2465686}. Since $\frac{1}{2}+\frac{a_{1} }{3} \le 1,$ the pair $\left( Y, e^1_{1}+\frac{1}{3}\omega_1 +\frac{a_1 + a_{12} + m}{3}F\right)  $ is not log canonical at $ Q $. By the inversion of adjunction formula, we get $ \left( e^1_{1}, \left.\left(\frac{1}{3}\omega_1+\frac{a_1 + a_{12} + m}{3}F\right)\right|_{e^1_{1}}  \right)  $ is not log canonical at $ Q $. 
    This implies that
    \begin{align*}
        3 &<e^1_{1} \cdot ( \omega_1 + (a_1 + a_{12} +m)F)\\ 
        &= e^1_{1}\cdot (d_1 - a_1e^1_1 - a_{12}h^1_{12} - a_{13}h^1_{13} - a_{14}h^1_{14}- a_{15}h^1_{15}) + (a_1 + a_{12} +m)\\
        &=(1 - (a_1 + a_{12} + m)) + 2a_{1} - ( a_{13} + a_{14} + a_{15} ) + (a_1 + a_{12} +m).
    \end{align*}

    \noindent
    Thus we get 
    $$2 + a_{13} + a_{14} + a_{15} < 2a_{1}.$$ 
    
    \noindent
    Together with the inequalities $a_{1} \leq 1 + a_{1k}$ for $k=3,4,5$, we achieve that
    $$2 + 3a_1 - 3 \leq 2 +  a_{13} + a_{14} + a_{15}  < 2a_{1}.$$

    \noindent
    So we get $a_1 < 1$ which contradicts to the inequality $2 +  a_{13} + a_{14} + a_{15} < 2a_{1}$.\\

    \noindent
    \textbf{Case 2.4}: $\varphi(P) \in h_{ij}\cap h_{lk}$ for some $ i \in \{ 1,2,3 \}$, $j \in \{ 1,2,3,4,5 \} \setminus \{ i \}$, $l,k \in \{ 1,2,3,4,5 \} \setminus \{ i,j \}$, and $(l,k) \neq (4,5)$. \\
    
    \noindent
    Similarly to Case 2.3. \end{proof}

   \end{Proposition}

    \subsubsection{$K^2 = 4 $ nodal type}
    The aim of this section is to prove that $ \glct \left( X,K_{X}\right) = \frac{1}{2} $, where $X$ is a Burniat surface with $K_X^2=4$ of nodal type (see Section \ref{Construction_of_Burniat_surface_with_K^2=4_nodal_type}). Since $\glct(X,K_{X})=2\glct(X, 2K_{X})$, it is sufficient to show that $\glct(X, 2K_{X}) = \frac{1}{4}$. We divide the proof of this result into the following two propositions.

    \begin{Proposition}\label{upperboundofglct}  $ \glct \left( X,2K_{X}\right) \le \frac{1}{4} $
   	\end{Proposition}
    
   \begin{proof}
    Let $ P\in X $ be such that $ \varphi\left( P\right) \in  e_3\cap h_{23}$. Consider the following divisor:
   	\begin{align*}
   		e:=2h_{23} + e_2 + e_3 + h_{145} \sim 3l-e_1 - e_2 - e_3 - e_4 - e_5.
   	\end{align*}
   \noindent
   The divisor $E:= \varphi^{*}\left( e\right) =4H_{23} + 2E_2 + 2E_3 + 2H_{145} \sim 2K_{X}$. This implies that $ \lct\left( X,E\right) = \frac{1}{4} $ at $P$ by Lemma \ref{comarelct}. So we get $ \glct\left( X,2K\right)  \le  \frac{1}{4}$.
   \end{proof}

    \begin{Proposition}\label{lowerboundofglct_nodal type_K^2=4}
   	$ \glct \left( X,2K_{X}\right) \ge \frac{1}{4} $.
   \end{Proposition}
	\begin{proof}
	Suppose on the contrary that $ \glct \left( X,2K_{X}\right) < \frac{1}{4} $. There exists an effective $ \mathbb{Q} $-Cartier divisor $ D \sim_{\mathbb{Q}} 2K_{X}$ such that $ \left( X, \frac{1}{4}D\right)  $ is not log canonical at some point $ P \in X $. By \cite[Proposition 9.5.11]{MR2095471}, we have that $\mult_P\left( D\right) >4$.

    \noindent
    We denote by $ d:= \varphi\left( D\right)  $. By Lemma \ref{log_canonical_between_cover}, $ \left( Y_{5}', \frac{1}{4}d +\frac{1}{2}B\right)  $ is not log canonical at $ \varphi\left( P\right)$ where $B:=B_1+B_2+B_3$.\\ 
	
    \noindent
    \textbf{Case 1}: Assume that $\varphi\left( P\right) \notin B  $. Since $ \left( Y_{5}',  \frac{1}{4}d \right)  $ is not log canonical at $ \varphi\left( P\right)  $,  $ \glct\left( Y_{5}', d\right) < \frac{1}{4} $. We notice that $ d \sim_\mathbb{Q} -K_{Y_{5}'} $ and that $ \glct\left( Y_4, -K_{Y_4}\right) = \frac{1}{2} $ by \cite[Theorem 1.7]{MR2465686}. Moreover, since $Y_{5}'$ is the blow-up of $Y_4$ at a point where $Y_4$ has only canonical singularities (in fact, $Y_4$ is smooth), $\glct\left( Y_{5}', -K_{Y_{5}'}\right) \ge \glct\left( Y_4, -K_{Y_4}\right) = \frac{1}{2}$ (c.f. \cite[Lemma 2.7]{MR2747875}). We get a contradiction.\\

    \noindent
    \textbf{Case 2}: Assume that $\varphi\left( P\right) \in B  $. We show that we get a contradiction.\\

    \noindent
	\textbf{Case 2.1}: $P \in E_1 \cup E_2 \cup E_3$.\\

    \noindent
    \noindent For simplicity, we show that we get a contradiction if $P \in E_1 $. The same proof works for $P \in E_2 \cup E_3$.\\

    \noindent
    Assume that $P \in E_1\setminus H_{12} $. We write
	\begin{align*}
		D = a_1E_1 + a_{12}H_{12} + \Omega,
	\end{align*}
	\noindent
	$ a_1, a_{12}\ge 0 $ and $\Omega$ is an effective $\mathbb{Q}$-divisor satisfying $ E_1,H_{12} \nsubseteq \support\left( \Omega\right) $. 
	
	\noindent
	We have
	\begin{align*}
		8 = D\cdot\varphi^{*}\left( t_1\right) = \left( a_1E_1 + a_{12}H_{12}  + \Omega\right)\cdot \varphi^{*}\left( t_1\right) \ge 2a_1.
	\end{align*}
	\noindent
	So $ a_1 \le 4 $. Since $ \left( X, \frac{1}{4}D\right)  $ is not log canonical at $ P $ and $ \frac{a_1}{4} \le 1$, $ \left( X, E_1 +\frac{1}{4}\Omega\right)  $ is not log canonical at $ P $. By the inversion of adjunction formula, we get $ \left( E_1, \left.\frac{1}{4}\Omega\right|_{E_1}  \right)  $ is not log canonical at $ P $. So we obtain 
	$$ \mult_P\left( \Omega|_{E_1} \right) >4. $$
	\noindent
	This yields that 
	\begin{align*}
		2 + a_1 - a_{12}  &=\left(  D- a_1E_1 - a_{12}H_{12} \right)\cdot E_1 \\ 
        &= \Omega\cdot E_1 \geq \mult_P\left( \Omega|_{E_1} \right) >4.
	\end{align*}
	
    \noindent
    On the other hand, we have
	\begin{align*}
		2 = D\cdot H_{12} = \left( a_1E_1 + a_{12}H_{12} + \Omega\right)\cdot H_{12} \ge a_1 - a_{12}.
	\end{align*}
	
    \noindent
	These above inequalities imply that $ 2 < a_1 - a_{12} \le 2$, a contradiction.\\
    
    \noindent
    Assume that $P \in E_1 \cap H_{12} $. Write
    $$D = a_1E_1 + a_{12}H_{12} + a_{34}H_{34} + \Omega,$$
    \noindent
    where $a_1 \ge 0, a_{12}, a_{34} \ge 0$ and $E_1, H_{12}, H_{34} \nsubseteq \textup{supp}(\Omega)$. 
    \begin{align*}
        8=D\cdot \varphi^{*}(t_3) \ge (a_1E_1 + a_{12}H_{12} + a_{34}H_{34} + \Omega)\cdot\varphi^{*}(t_3) \ge2a_{12}
    \end{align*}
    \noindent
    So $ a_{12} \le 4 $. Since $ \left( X, \frac{1}{4}D\right)  $ is not log canonical at $ P $ and $ \frac{a_{12}}{4} \le 1$, $ \left( X, H_{12} +\frac{1}{4}(a_1E_1 +\Omega)\right)  $ is not log canonical at $ P $. By the inversion of adjunction formula, we get $ \left( H_{12}, \left.\frac{1}{4}(a_1E_1 +\Omega)\right|_{H_{12}}  \right)  $ is not log canonical at $ P $. So we obtain 
    \begin{align*}
    2+a_{12}-a_{34}&=H_{12}\cdot (D-a_{12}H_{12}-a_{34}H_{34})=H_{12}\cdot(a_1E_1+\Omega)\\
&\geq \mult_P\left( H_{12} \right)\mult_P\left( a_1E_1+\Omega \right) >4
    \end{align*}

    \noindent
    On the other hand, we have
    \begin{align*}
        0\leq \Omega \cdot H_{34} = (D - a_1E_1 - a_{12}H_{12} -a_{34}H_{34})\cdot H_{34} = 2 - a_{12} + a_{34},
    \end{align*}
    which contradicts to the last inequalities.\\

    \noindent
    \textbf{Case 2.2}: $P \in  H_{ij}\setminus E_i $ for some $i \in \{1,2,3 \}$ and $ j \in \{1,2,3,4,5 \}$ with $i \neq j$ and $\left(i,j\right) \notin \{\left(1,4\right), \left(1,5\right)\}$.\\
    
    \noindent
    \noindent For simplicity, we show that we get a contradiction in the case where $\left(i,j\right) = \left(1,2\right)$. The same proof works for other $\left(i,j\right)$.\\

    \noindent
    Suppose that $P \in  H_{12}\setminus E_1 $. We write
	\begin{align*}
		D = a_{12}H_{12} + a_1E_1 + \Omega,
	\end{align*}
	\noindent
	$ a_1, a_{12} \ge 0 $ and $\Omega$ is an effective $\mathbb{Q}$-divisor satisfying $ E_1,H_{12} \nsubseteq \support\left( \Omega\right) $.\\

    \noindent
	We have
	\begin{align*}
		8 = D\cdot\varphi^{*}\left( t_3\right) = \left( a_{12}H_{12} + a_1E_1 + \Omega\right)\cdot \varphi^{*}\left( t_3\right) \ge 2a_{12}.
	\end{align*}

    \noindent
    So $ a_{12} \le 4 $. Since $ \left( X, \frac{1}{4}D\right)  $ is not log canonical at $ P $ and $ \frac{a_{12}}{4} \le 1$, $ \left( X, H_{12} +\frac{1}{4}\Omega\right)  $ is not log canonical at $ P $. By the inversion of adjunction formula, we get $ \left( H_{12}, \left.\frac{1}{4}\Omega\right|_{H_{12}}  \right)  $ is not log canonical at $ P $. So we obtain 
	$$ \mult_P\left( \Omega|_{H_{12}} \right) >4. $$
	\noindent
	This yields that 
	\begin{align*}
		2 + a_{12} - a_{1}  &=\left(  D- a_{12}H_{12} - a_1E_1\right)\cdot H_{12} \\ 
        &= \Omega\cdot H_{12} \geq \mult_P\left( \Omega|_{H_{12}} \right) >4.
	\end{align*}

    \noindent
    On the other hand, we have
	\begin{align*}
		2 = D\cdot E_{1} = \left(a_{12}H_{12} + a_1E_1 + \Omega\right)\cdot E_{1} \ge a_{12} -a_1.
	\end{align*}
	
    \noindent
	These above inequalities imply that $ a_1 + 2 < a_{12} \le a_1 + 2$, a contradiction.\\

    \noindent
    \textbf{Case 2.3}: $P \in  H_{145}\setminus E_1 $.\\
    
    \noindent
    We write
	\begin{align*}
		D = a_1E_1+  a_{4}\Tilde{E}_{4} + a_{5}\Tilde{E}_{5}+a_{145}H_{145}+\Omega,
	\end{align*}
	\noindent
	$ a_1,a_4, a_5,_{145} \ge 0 $ and $\Omega$ is an effective $\mathbb{Q}$-divisor satisfying $ E_1,\Tilde{E}_4, \Tilde{E}_5,H_{145}\nsubseteq \support\left( \Omega\right)$.\\

    \noindent
	We have
	\begin{align*}
		8 = D\cdot\varphi^{*}\left( t_2\right) = \left( a_1E_1+  a_{4}\Tilde{E}_{4} + a_{5}\Tilde{E}_{5}+a_{145}H_{145}+\Omega\right)\cdot \varphi^{*}\left( t_2\right) \ge 2a_{145}.
	\end{align*}

    \noindent
    So $ a_{145} \le 4 $. 

Suppose now that $P\in H_{145}\setminus (\Tilde{E}_4\cup \Tilde{E}_5)$. Since $ \left( X, \frac{1}{4}D\right)  $ is not log canonical at $ P $ and $ \frac{a_{145}}{4} \le 1$ and $P\notin E_1\cup \Tilde{E}_4\cup \Tilde{E}_5$, $ \left( X, H_{145} +\frac{1}{4}\Omega\right)  $ is not log canonical at $ P $.  By the inversion of adjunction formula, we get $ \left( H_{145}, \left.\frac{1}{4}\Omega\right|_{H_{145}}  \right)  $ is not log canonical at $ P $. So we obtain 
	$$ \mult_P\left( \Omega|_{H_{145}} \right) >4. $$
	\noindent
	This yields that 
	\begin{align*}
	 2a_{145} -a_1- 2a_{4}- 2a_{5} &=\left(  D- a_{145}H_{145} -a_1E_1- a_4\Tilde{E}_4 - a_5\Tilde{E}_5\right)\cdot H_{145} \\ 
        &= \Omega\cdot H_{145} \geq \mult_P\left( \Omega|_{H_{145}} \right) >4.
	\end{align*}
Since $a_1\ge 0$, we have $a_{145}-a_4-a_5>2.$\\

    \noindent
    On the other hand, we have
	\begin{align*}
		4 = D\cdot \Tilde{E}_4 = \left(a_1E_1 + a_4\Tilde{E}_4 + a_5\Tilde{E}_5 +a_{145}H_{145}+ \Omega\right)\cdot \Tilde{E}_4 \ge 2a_{145}- 4a_{4}
	\end{align*}
    \noindent
    and
        \begin{align*}
		4 = D\cdot \Tilde{E}_5 = \left( a_1E_1+a_4\Tilde{E}_4 + a_5\Tilde{E}_5 + a_{145}H_{145} +\Omega\right)\cdot \Tilde{E}_5 \ge 2a_{145}- 4a_{5}.
	\end{align*}
    \noindent
    So we get
    \begin{align*}
        8 \ge 4a_{145}- 4a_{4}- 4a_{5}
    \end{align*}
    \noindent
    which contradicts the above inequality
   \[
        a_{145}- a_{4}- a_{5} > 2.
    \]

Finally, suppose that $P\in H_{145}\cap \Tilde{E}_i$ for $i=4$ (resp.  $i=5$). Define $j=5$ (resp. $j=4$). 
Since $ \left( X, \frac{1}{4}D\right)  $ is not log canonical at $ P $ and $ \frac{a_{145}}{4} \le 1$ and $P\notin E_1\cup \Tilde{E}_j$, $\left( X, H_{145} +\frac{1}{4}(a_i\Tilde{E}_i+\Omega)\right)$ is not log canonical at $ P $.  By the inversion of adjunction formula, we get $ \left( H_{145}, \left.\frac{1}{4}(a_i\Tilde{E}_i+\Omega)\right|_{H_{145}}  \right) $ is not log canonical at $ P $. So we obtain 
	$$ \mult_P\left( (a_i\Tilde{E}_i+\Omega)|_{H_{145}} \right) >4. $$
	\noindent
	This yields that 
	\begin{align*}
		-a_1- 2a_{j}+2a_{145}  &=\left(  D- a_1E_1- a_j\Tilde{E}_j-a_{145}H_{145} \right)\cdot H_{145} \\ 
        &= (a_i\Tilde{E}_i+\Omega)\cdot H_{145} \geq \mult_P\left( a_i\Tilde{E}_i+\Omega|_{H_{145}} \right) >4.
	\end{align*}

    \noindent
    On the other hand, we have
	\begin{align*}
		4 = D\cdot \Tilde{E}_j = \left(a_1E_1+a_{145}H_{145} + a_4\Tilde{E}_4 + a_5\Tilde{E}_5 + \Omega\right)\cdot \Tilde{E}_j \ge 2a_{145}- 4a_{j}.
	\end{align*}
    \noindent
    
    Recall that     \begin{align*}
		2 = D\cdot E_1 = \left(a_1E_1+a_{145}H_{145} + a_4\Tilde{E}_4 + a_5\Tilde{E}_5 + \Omega\right)\cdot E_1 \ge -a_1+a_{145}.
	\end{align*}
    \noindent
It follows from the inequalities above that we obtain the desired contradiction that 
    \begin{align*}
        4=2+2\geq (a_{145}-2a_{j})+(-a_1+a_{145})=-a_1-2a_j+2a_{145}>4.
    \end{align*}
    
    \noindent
    \textbf{Case 2.4}: $P \in T_{11}$.\\
    
    \noindent
    We write
	\begin{align*}
		D = a_{11}T_{11} + \Omega,
	\end{align*}
	\noindent
	$ a_{11} \ge 0 $ and $ T_{11} \nsubseteq \support\left( \Omega\right) $. We have
	\begin{align*}
		8 = D\cdot\varphi^{*}\left( t_2\right) \ge 2a_{11}.
	\end{align*}
	\noindent
	So $ a_{11} \le 4 $. Since $ \left( X, \frac{1}{4}D\right)  $ is not log canonical at $ P $ and $ \frac{a_{11}}{4} \le 1$, $ \left( X, T_{11} +\frac{1}{4}\Omega\right)  $ is not log canonical at $ P $. By the inversion of adjunction formula $ \left( T_{11}, \left.\frac{1}{4}\Omega\right|_{T_{11}}  \right)  $ is not log canonical at $ P $. So 
	$$ T_{11}\cdot\Omega \ge \mult_P\left( \Omega|_{T_{11}} \right) >4. $$
	
	\noindent
	However $ 4 = T_{11}\cdot \left( D - a_{11}T_{11}\right) = T_{11}\cdot \Omega $ which is a contradiction.
    \end{proof}
    
\subsection{Global log canonical threshold of the secondary Burniat surfaces with $ K^2 = 3$.}

The aim of this section is to prove that $ \glct \left( X,K_{X}\right) = \frac{2}{3} $, where $X$ is a Burniat surface with $K_X^2=3$ (see Section \ref{Construction_of_Burniat_surface_with_K^2=3}). Since $\glct(X,K_{X})=2\glct(X, 2K_{X})$, it is sufficient to show that $\glct(X, 2K_{X}) = \frac{1}{3}$. We divide the proof of this result into the following two propositions.

\begin{Proposition}\label{upperboundofglct_K^2=3}  $ \glct \left( X,2K_{X}\right) \le \frac{1}{3} $
   	\end{Proposition}
    
   \begin{proof}
   	Let $ P\in X $ be such that $ \varphi\left( P\right) \in  e_1\cap h_{145}$. Consider the following divisor:
    {\small
   	\begin{align*}
   		e:=3h_{145} + h_{246}+h_{356}+h_{23}+e_1+2e_4+2e_5\sim 6l-2e_1 - 2e_2 - 2e_3 - 2e_4 - 2e_5-2e_6.
   	\end{align*}
    }
   \noindent
   The divisor $E:= \varphi^{*}\left( e\right) =6H_{145} + 2H_{246} + 2H_{356} + 2H_{23}+ 2E_{1}+2\tilde{E}_4+2\tilde{E}_5  \sim 4K_{X}$. This implies that $ \lct\left( X,E\right) = \frac{1}{6} $ at $P$ by Lemma \ref{comarelct}. So we get $$ \glct\left( X,2K\right) = 2\glct\left( X,4K\right) \le 2.\frac{1}{6} = \frac{1}{3}. $$ 
   \end{proof}

   \begin{Proposition}\label{lowerboundofglct_K^2=3}
   	$ \glct \left( X,2K_{X}\right) \ge \frac{1}{3} $.
   \end{Proposition}
   \begin{proof}
Suppose on the contrary that $ \glct \left( X,2K_{X}\right) < \frac{1}{3} $. There exists an effective $ \mathbb{Q} $-Cartier divisor $ D \sim_{\mathbb{Q}} 2K_{X}$ such that $ \left( X, \frac{1}{3}D\right)  $ is not log canonical at some point $ P \in X $. By \cite[Proposition 9.5.11]{MR2095471}, we have that $\mult_P\left( D\right) >3$.

    \noindent
    We denote by $ d:= \varphi\left( D\right)  $. By Lemma \ref{log_canonical_between_cover}, $ \left( Y_{6}, \frac{1}{3}d +\frac{1}{2}B\right)  $ is not log canonical at $ \varphi\left( P\right)$ where $B:=B_1+B_2+B_3$.\\ 
	
    \noindent
    \textbf{Case 1}: Assume that $\varphi\left( P\right) \notin B  $. Since $ \left( Y_{6},  \frac{1}{3}d \right)  $ is not log canonical at $ \varphi\left( P\right)  $,  $ \glct\left( Y_{6}, d\right) < \frac{1}{3} $. We notice that $ d \sim_\mathbb{Q} -K_{Y_{6}} $ and that $ \glct\left( Y_4, -K_{Y_4}\right) = \frac{1}{2} $ by \cite[Theorem 1.7]{MR2465686}. Moreover, since $Y_{6}$ is the blow-up of $Y_4$ at two points where $Y_4$ has only canonical singularities (in fact, $Y_4$ is smooth), $\glct\left( Y_{6}, -K_{Y_{6}}\right) \ge \glct\left( Y_4, -K_{Y_4}\right) = \frac{1}{2}$ (c.f. \cite[Lemma 2.7]{MR2747875}). We get a contradiction.\\

    \noindent
    \textbf{Case 2}: Assume that $\varphi\left( P\right) \in B  $.\\

    \noindent
    \textbf{Case 2.1}:   $\varphi\left( P\right) \in (e_1\cup e_2\cup e_3) \setminus (h_{12}\cup h_{13}\cup h_{23}\cup h_{16}\cup h_{25}\cup h_{34}\cup h_{145}\cup h_{246} \cup h_{356} )$.\\
    
    \noindent Suppose that  $\varphi\left( P\right) \in e_{1} \setminus (h_{12} \cup h_{13} \cup h_{16}\cup h_{145})$. \\  

    \noindent
    Write
    $$d = a_1e_1 +a_{16}h_{16}+a_{145}h_{145}+a_{23}h_{23}+  \omega,$$
    \noindent
    where $a_1, a_{16}, a_{145}, a_{23} \ge 0$ and $e_1, h_{16}, h_{145}, h_{23} \nsubseteq \textup{supp}(\omega)$. Note that  
    \begin{align*}
        1&=d\cdot h_{16}\geq a_1-a_{16}+a_{23},\\
        1&=d\cdot h_{23}\geq a_{16}+a_{145}-a_{23},\\
        0&=d\cdot h_{145}\geq a_1+a_{23}-2a_{145}.
    \end{align*}
    This yields 
    \[
        2=d\cdot (h_{16}+h_{23})\geq a_1+a_{145}\geq \frac{3}{2}a_1+\frac{1}{2}a_{23}\geq \frac{3}{2}a_1. \]
        Thus one has $a_1\leq \frac{4}{3}<\frac{3}{2}$.
    \noindent
    Because $(Y_6,\frac{1}{3}d+\frac{1}{2}e_1)$ is not log canonical at $\varphi(P)$ and $ \frac{a_{1}}{3}+\frac{1}{2} < 1$, the pair $ \left( Y_6, e_{1} +\frac{1}{3}\omega\right)  $ is not log canonical at $ \varphi(P) $. By the inversion of adjunction formula, we get $ \left( e_{1}, \left.\frac{1}{3}\omega\right|_{e_{1}}  \right)  $ is not log canonical at $ \varphi(P) $. 
    This yields that 
    \begin{align*}
		1+ a_{1}-a_{16}-a_{145}  &=\left(  d - a_{1}e_{1}-a_{16}h_{16}-a_{23}h_{23}-a_{145}h_{145} \right)\cdot e_{1} \\ 
        &= \omega\cdot e_{1} \geq \mult_{\varphi(P)}\left( \omega|_{e_{1}} \right) >3.
    \end{align*}
    Thus, $a_{1}-a_{16}>2+a_{145}$. However, 
    this contradicts to the above inequality $1=d\cdot h_{16}\geq a_1-a_{16}+a_{23}$.
\\

 \noindent 
 The same proof works for $\varphi\left( P\right) \in e_{2} \setminus (h_{12} \cup h_{23} \cup h_{25}\cup h_{246})$ and $\varphi\left( P\right) \in e_{3} \setminus (h_{13} \cup h_{23} \cup h_{34}\cup h_{356})$.\\

    \noindent
    \textbf{Case 2.2}:   $\varphi\left( P\right) \in (h_{12}\cup h_{13}\cup h_{23}) \setminus (e_1\cup e_2\cup e_3\cup h_{16}\cup h_{25}\cup h_{34}\cup h_{145}\cup h_{246} \cup h_{356} )$.\\
    
    \noindent The argument is similar to Case 2.1. \\
 
    \noindent
    \textbf{Case 2.3}: $\varphi\left( P\right) \in h_{16} \setminus (e_1\cup h_{23}\cup h_{25}\cup h_{34})$, $\varphi\left( P\right) \in h_{25} \setminus (e_2\cup h_{13}\cup h_{16}\cup h_{34})$ or $\varphi\left( P\right) \in h_{34} \setminus (e_3\cup h_{12}\cup h_{16}\cup h_{25})$.\\
    
    \noindent The argument is similar to Case 2.1. We would hide the following argument later.
    
    \noindent Suppose that  $\varphi\left( P\right) \in h_{16} \setminus (e_1\cup h_{23}\cup h_{25}\cup h_{34})$. \\  

    \noindent
    Write
    $$d =  a_{16}h_{16}+a_{25}h_{25}+a_{34}h_{34}+  \omega,$$
    \noindent
    where $ a_{16}, a_{25}, a_{34} \ge 0$ and $ h_{16}, h_{25}, h_{34}  \nsubseteq \textup{supp}(\omega)$. Note that  
    \begin{align*}
        1&=d\cdot h_{25}\geq a_{16}-a_{25}+a_{34},\\
        1&=d\cdot h_{34}\geq a_{16}+a_{25}-a_{34}.
    \end{align*}
        Thus one has $a_{16}\leq 1<\frac{3}{2}$.
    \noindent
    Because $(Y_6,\frac{1}{3}d+\frac{1}{2}h_{16})$ is not log canonical at $\varphi(P)$ and $ \frac{a_{16}}{3}+\frac{1}{2} < 1$, the pair $ \left( Y_6, h_{16} +\frac{1}{3}\omega\right)  $ is not log canonical at $ \varphi(P) $. By the inversion of adjunction formula, we get $ \left( h_{16}, \left.\frac{1}{3}\omega\right|_{h_{16}}  \right)  $ is not log canonical at $ \varphi(P) $. 
    This yields that 
    \begin{align*}
		1+ a_{16}-a_{25}-a_{34} &=\left(  d - a_{16}h_{16}-a_{25}h_{25}-a_{34}h_{34}\right)\cdot h_{16} \\ 
        &= \omega\cdot h_{16} \geq \mult_{\varphi(P)}\left( \omega|_{h_{16}} \right) >3.
    \end{align*}
    In particular, $a_{16}>2$. However, 
    this contradicts to the above inequality $a_{16}\leq 1$.
\\

 \noindent 
 The same proof works for $\varphi\left( P\right) \in h_{25} \setminus (e_2\cup h_{13}\cup h_{16}\cup h_{34})$ and  $\varphi\left( P\right) \in h_{34} \setminus (e_3\cup h_{12}\cup h_{16}\cup h_{25})$.\\

    \noindent
    \textbf{Case 2.4}: $\varphi\left( P\right) \in (h_{145}\cup h_{246}\cup h_{356}) \setminus (e_1\cup e_2\cup e_3\cup h_{12}\cup h_{13}\cup h_{23} )$.\\
    
    
    \noindent Suppose that  $\varphi\left( P\right) \in h_{145} \setminus (e_1 \cup h_{23})$. \\  

    \noindent
    Write
    $$d = a_1e_1 +a_{16}h_{16}+a_{23}h_{23}+a_{145}h_{145}+  \omega,$$
    \noindent
    where $a_1, a_{16}, a_{23}, a_{145} \ge 0$ and $e_1, h_{16}, h_{23}, h_{145}  \nsubseteq \textup{supp}(\omega)$. Note that  
    \begin{align*}
        1&=d\cdot e_{1}\geq -a_{1}+a_{16}+a_{145},\\
        1&=d\cdot h_{23}\geq a_{16}-a_{23}+a_{145},\\
        1&=d\cdot h_{16}\geq a_1-a_{16}+a_{23},\\
        0&=d\cdot h_{145}\geq a_1+a_{23}-2a_{145}.
    \end{align*}
    This yields 
    \[
        2=d\cdot (h_{16}+h_{23})\geq a_1+a_{145}\geq a_{16}+2a_{145}-1\geq 2a_{145}-1. \]
        Thus one has $a_{145}\leq \frac{3}{2}$.
    Because $(Y_6,\frac{1}{3}d+\frac{1}{2}h_{145})$ is not log canonical at $\varphi(P)$ and $ \frac{a_{145}}{3}+\frac{1}{2} \leq 1$, the pair $ \left( Y_6, h_{145} +\frac{1}{3}\omega\right)  $ is not log canonical at $ \varphi(P) $. By the inversion of adjunction formula, we get $ \left( h_{145}, \left.\frac{1}{3}\omega\right|_{h_{145}}  \right)  $ is not log canonical at $ \varphi(P) $. 
    This yields that 
    \begin{align*}
		-a_{1}-a_{23}+2a_{145}  &=\left(  d - a_{1}e_{1}-a_{16}h_{16}-a_{23}h_{23}-a_{145}h_{145} \right)\cdot h_{145} \\ 
        &= \omega\cdot h_{145} \geq \mult_{\varphi(P)}\left( \omega|_{h_{145}} \right) >3.
    \end{align*}
    In particular, $a_{145}>\frac{3}{2}$ which contradicts to the above inequality $a_{145}\leq \frac{3}{2}$.
\\

 \noindent 
 The same proof works for $\varphi\left( P\right) \in h_{246} \setminus (e_2 \cup  h_{13})$ and $\varphi\left( P\right) \in h_{356} \setminus (e_3 \cup  h_{12})$.\\


\noindent
    \textbf{Case 2.5}: $\varphi\left( P\right) \in h_{16}\cap h_{25}$, $\varphi\left( P\right) \in h_{16}\cap h_{34}$ or $\varphi\left( P\right) \in h_{25}\cap h_{34}$.\\

  \noindent 
    Suppose that  $\varphi\left( P\right) \in h_{16} \cap h_{25} $. \\   

    \noindent
    Write
    $$d =  a_{16}h_{16}+a_{25}h_{25}+a_{34}h_{34}+  \omega,$$
    \noindent
    where $ a_{16}, a_{25}, a_{34} \ge 0$ and $ h_{16}, h_{25}, h_{34}  \nsubseteq \textup{supp}(\omega)$. Note that  
    \begin{align*}
        1&=d\cdot h_{25}\geq a_{16}-a_{25}+a_{34},\\
        1&=d\cdot h_{34}\geq a_{16}+a_{25}-a_{34}.
    \end{align*}
        Thus one has $a_{16}\leq 1$.
Because $(Y_6,\frac{1}{3}d+\frac{1}{2}h_{16}+\frac{1}{2}h_{25})$ is not log canonical at $\varphi(P)$ and $ \frac{a_{16}}{3}+\frac{1}{2} < 1$, the pair $ \left( Y_6, h_{16} +\frac{1}{3}(a_{25}h_{25} +\omega) +\frac{1}{2}h_{25}\right)  $ is not log canonical at $ \varphi(P) $. By the inversion of adjunction formula, we get $$ \left( h_{16}, \left.\left(\frac{1}{3}\left(a_{25}h_{25} +\omega\right) +\frac{1}{2}h_{25}\right)\right|_{h_{16}}  \right)  $$ is not log canonical at $ \varphi(P) $. 
    This yields that 
    \begin{align*}
		1+a_{16}-a_{34}  &=\left(  d - a_{16}h_{16}-a_{34}h_{34}\right)\cdot h_{16} = (a_{25}h_{25} +\omega)\cdot h_{16}  >\frac{3}{2}.
    \end{align*}
    Thus, $a_{16}>\frac{1}{2} + a_{34}$.
    Similar argument gives $a_{25}\leq 1$ and $a_{25}>\frac{1}{2}+a_{34}$.
    Together with the inequality $$1=d\cdot h_{34}\geq a_{16}+a_{25}-a_{34},$$ we obtain that 
    $$a_{16}>\frac{1}{2} + a_{34}\geq \frac{1}{2}+(a_{16}+a_{25}-1).$$

    \noindent
    So we have
    $a_{25}\leq \frac{1}{2}$ which contradicts to $a_{25}>\frac{1}{2}+a_{34}\geq \frac{1}{2}$.\\

\noindent The same proof works for $\varphi\left( P\right) \in h_{16}\cap h_{34}$ and $\varphi\left( P\right) \in h_{25}\cap h_{34}$.\\

  \noindent
    \textbf{Case 2.6}: $\varphi\left( P\right) \in e_{1} \cap (h_{12}\cup h_{13})$, $\varphi\left( P\right) \in e_{2} \cap (h_{12}\cup h_{23})$ or $\varphi\left( P\right) \in e_{3} \cap (h_{13}\cup h_{23})$.\\

    \noindent 
    Suppose that  $\varphi\left( P\right) \in e_{1} \cap h_{12} $. \\  

    \noindent
    Write
    $$d = a_1e_1 +a_{12}h_{12}+a_{16}h_{16}+a_{23}h_{23}+a_{34}h_{34}+a_{145}h_{145}+a_{356}h_{356}+  \omega,$$
    \noindent
    where $a_1, a_{12}, a_{16}, a_{23}, a_{34}, a_{145}, a_{356} \ge 0$ and $e_1, h_{12}, h_{16}, h_{23}, h_{34}, h_{145}, h_{356} \nsubseteq \textup{supp}(\omega)$. Note that  
    \begin{align*}
        1&=d\cdot h_{16}\geq a_1-a_{16}+a_{23} +a_{34},\\
        1&=d\cdot h_{23}\geq a_{16}+a_{145}-a_{23},\\
        0&=d\cdot h_{145}\geq a_1+a_{23}-2a_{145}.
    \end{align*}
    This yields 
    \[
        2=d\cdot (h_{16}+h_{23})\geq a_1+a_{145}+a_{34}\geq \frac{3}{2}a_1+\frac{1}{2}a_{23}+a_{34}\geq \frac{3}{2}a_1. \]
    Thus one has $a_1\leq \frac{4}{3}$.
    \noindent
    Because $(Y_6,\frac{1}{3}d+\frac{1}{2}e_1+\frac{1}{2}h_{12})$ is not log canonical at $\varphi(P)$ and $ \frac{a_{1}}{3}+\frac{1}{2} < 1$, the pair $ \left( Y_6, e_{1} +\frac{1}{3}(a_{12}h_{12} +\omega) +\frac{1}{2}h_{12}\right)  $ is not log canonical at $ \varphi(P) $. By the inversion of adjunction formula, we get $ \left( e_{1}, \left.\left(\frac{1}{3}\left(a_{12}h_{12} +\omega\right) +\frac{1}{2}h_{12}\right)\right|_{e_{1}}  \right)  $ is not log canonical at $ \varphi(P) $. 
    This yields that 
    \begin{align*}
		1+ a_{1}-a_{16}-a_{145}  &=\left(  d - a_1e_1-a_{16}h_{16}-a_{23}h_{23}-a_{34}h_{34}-a_{145}h_{145}-a_{356}h_{356} \right)\cdot e_{1} \\ 
        &= (a_{12}h_{12} +\omega)\cdot e_{1}  >\frac{3}{2}.
    \end{align*}
    Thus, $a_{1}>\frac{1}{2} + a_{16}+a_{145}$. Together with the inequality $0\geq a_1+a_{23}-2a_{145}$, we obtain that 
    $$a_{1}>\frac{1}{2} + a_{16}+\frac{a_1}{2} + \frac{a_{23}}{2}. $$

    \noindent
    So we have
    $$a_1 > 1.$$
    
    \noindent
    Note that  
    \begin{align*}
        1&=d\cdot h_{16}\geq a_1-a_{16}+a_{23}+a_{34},\\
        1&=d\cdot h_{34}\geq a_{12}+a_{16}-a_{34}.
    \end{align*}
    This yields 
    $$a_1 + a_{12} + a_{23} \le 2.$$
    \noindent
    So we get $a_{12} < 1$. Because $(Y_6,\frac{1}{3}d+\frac{1}{2}e_1+\frac{1}{2}h_{12})$ is not log canonical at $\varphi(P)$ and $ \frac{a_{12}}{3}+\frac{1}{2} < 1$, the pair $ \left( Y_6, h_{12} +\frac{1}{3}(a_{1}e_{1} +\omega) +\frac{1}{2}e_{1}\right)  $ is not log canonical at $ \varphi(P) $. By the inversion of adjunction formula, we get $ \left( h_{12}, \left.\left(\frac{1}{3}\left(a_{1}e_{1} +\omega\right) +\frac{1}{2}e_{1}\right)\right|_{h_{12}}  \right)  $ is not log canonical at $ \varphi(P) $. 
    This yields that 
    \begin{align*}
		1+ a_{12}-a_{34}-a_{356}  &=\left(  d - a_{12}h_{12}-a_{16}h_{16}-a_{23}h_{23}-a_{34}h_{34}-a_{145}h_{145}-a_{356}h_{356} \right)\cdot h_{12} \\ 
        &= (a_{1}e_{1} +\omega)\cdot h_{12}  >\frac{3}{2}.
    \end{align*}
    So, $a_{12}>\frac{1}{2} + a_{34}+a_{356}$. Together with the inequality $0 = d\cdot h_{356}\geq a_{12}-2a_{356}$, we obtain that 
    $$a_{12}>\frac{1}{2} + a_{34}+\frac{a_{12}}{2}. $$
    \noindent
    Thus one gets $a_{12} > 1$ which contradicts $a_{12} < 1$.\\

    \noindent The same proof works for the cases $\varphi\left( P\right) \in e_1\cap h_{13}$ and $\varphi\left( P\right) \in e_{2} \cap (h_{12}\cup h_{23})$ and $\varphi\left( P\right) \in e_{3} \cap (h_{13}\cup h_{23})$.\\
    
    \noindent
    \textbf{Case 2.7}: $\varphi\left( P\right) \in e_{1} \cap h_{16}$, $\varphi\left( P\right) \in e_{2} \cap h_{25}$ or $\varphi\left( P\right) \in e_{3} \cap h_{34}$.\\

\noindent 
    Suppose that $\varphi\left( P\right) \in e_{1} \cap h_{16} $. \\  

    \noindent
    Write
    $$d = a_1e_1 +a_{12}h_{12}+a_{16}h_{16}+a_{23}h_{23}+a_{34}h_{34}+a_{145}h_{145}+a_{356}h_{356}+  \omega,$$
    \noindent
    where $a_1, a_{12}, a_{16}, a_{23}, a_{34}, a_{145}, a_{356} \ge 0$ and $e_1, h_{12}, h_{16}, h_{23}, h_{34}, h_{145}, h_{356} \nsubseteq \textup{supp}(\omega)$. Note that  
    \begin{align*}
        1&=d\cdot h_{16}\geq a_1-a_{16}+a_{23} +a_{34},\\
        1&=d\cdot h_{23}\geq a_{16}+a_{145}-a_{23},\\
        0&=d\cdot h_{145}\geq a_1+a_{23}-2a_{145}.
    \end{align*}
    This yields 
    \[
        2=d\cdot (h_{16}+h_{23})\geq a_1+a_{145}+a_{34}\geq \frac{3}{2}a_1+\frac{1}{2}a_{23}+a_{34}\geq \frac{3}{2}a_1. \]
    Thus one has $a_1\leq \frac{4}{3}$.
    \noindent
    Because $(Y_6,\frac{1}{3}d+\frac{1}{2}e_1+\frac{1}{2}h_{16})$ is not log canonical at $\varphi(P)$ and $ \frac{a_{1}}{3}+\frac{1}{2} < 1$, the pair $ \left( Y_6, e_{1} +\frac{1}{3}(a_{16}h_{16} +\omega) +\frac{1}{2}h_{16}\right)  $ is not log canonical at $ \varphi(P) $. By the inversion of adjunction formula, we get $ \left( e_{1}, \left.\left(\frac{1}{3}\left(a_{16}h_{16} +\omega\right) +\frac{1}{2}h_{16}\right)\right|_{e_{1}}  \right)  $ is not log canonical at $ \varphi(P) $. 
    This yields that
    {\small
    \begin{align*}
		1+ a_{1}-a_{12}-a_{145}  &=\left(  d - a_1e_1-a_{12}h_{12}-a_{23}h_{23}-a_{34}h_{34}-a_{145}h_{145}-a_{356}h_{356} \right)\cdot e_{1} \\ 
        &= (a_{16}h_{16} +\omega)\cdot e_{1}  >\frac{3}{2}.
    \end{align*}
    }
    Thus, $a_{1}>\frac{1}{2} + a_{12}+a_{145}$. Together with the inequality $0\geq a_1+a_{23}-2a_{145}$, we obtain that 
    $$a_{1}>\frac{1}{2} + a_{12}+\frac{a_1}{2} + \frac{a_{23}}{2}. $$

    \noindent
    So we have
    $$a_1 > 1.$$

    \noindent
    Note that  
    \begin{align*}
        1&=d\cdot h_{12}\geq a_1-a_{12}+a_{34}+a_{356},\\
        1&=d\cdot h_{34}\geq a_{12}+a_{16}-a_{34}.
    \end{align*}
    This yields 
    $$a_1 + a_{16} + a_{356} \le 2.$$
    \noindent
    So we get $a_{16} < 1$. Because $(Y_6,\frac{1}{3}d+\frac{1}{2}e_1+\frac{1}{2}h_{16})$ is not log canonical at $\varphi(P)$ and $ \frac{a_{16}}{3}+\frac{1}{2} < 1$, the pair $ \left( Y_6, h_{16} +\frac{1}{3}(a_{1}e_{1} +\omega) +\frac{1}{2}e_{1}\right)  $ is not log canonical at $ \varphi(P) $. By the inversion of adjunction formula, we get $ \left( h_{16}, \left.\left(\frac{1}{3}\left(a_{1}e_{1} +\omega\right) +\frac{1}{2}e_{1}\right)\right|_{h_{16}}  \right)  $ is not log canonical at $ \varphi(P) $. 
    This yields that 
    {\small
    \begin{align*}
		1+ a_{16}-a_{23}-a_{34}  &=\left(  d - a_{12}h_{12}-a_{16}h_{16}-a_{23}h_{23}-a_{34}h_{34}-a_{145}h_{145}-a_{356}h_{356} \right)\cdot h_{16} \\ 
        &= (a_{1}e_{1} +\omega)\cdot h_{16}  >\frac{3}{2}.
    \end{align*}
    }
    So, $a_{16}>\frac{1}{2} + a_{23}+a_{34}$. Together with the inequality 
    $$1 = d\cdot h_{23}\geq a_{16}+a_{145} - a_{23} \geq a_{16} + \frac{a_1}{2} + \frac{a_{23}}{2} -a_{23},$$
    \noindent
    we obtain that 
    $$a_{16}>\frac{1}{2} + \left(a_{16} + \frac{a_1}{2} + \frac{a_{23}}{2} -1\right) + a_{34}. $$
    \noindent
    Thus one gets $a_{1} < 1$ which contradicts to $a_{1} > 1$.\\

    \noindent The same proof works for the cases $\varphi\left( P\right) \in e_{2} \cap h_{25}$ and  $\varphi\left( P\right) \in e_{3} \cap h_{34}$.\\

    \noindent
    \textbf{Case 2.8}: $\varphi\left( P\right) \in h_{12} \cap h_{34}$, $\varphi\left( P\right) \in h_{13} \cap h_{25}$ or $\varphi\left( P\right) \in h_{23} \cap h_{16}$.\\

    \noindent The arguments are similar to Case 2.7.\\

    \noindent
    \textbf{Case 2.9}: $\varphi\left( P\right) \in e_{1} \cap h_{145}$, $\varphi\left( P\right) \in e_{2} \cap h_{246}$ or $\varphi\left( P\right) \in e_{3} \cap h_{356}$.\\

\noindent Suppose that $\varphi\left( P\right) \in e_{1} \cap h_{145} $.  \\

 \noindent
    Write
    \begin{align*}
        d &= a_1e_1+a_4e_4+a_5e_5+a_{12}h_{12}+a_{13}h_{13}+a_{16}h_{16}+a_{23}h_{23}\\
        &\ \ \ \ +a_{145}h_{145}+  \omega
    \end{align*}
     \noindent
    where $a_1, a_4, a_5, a_{12},a_{13}, a_{16}, a_{23}, a_{145} \ge 0$ and $e_1, e_4, e_5, h_{12}, h_{13}, h_{16}$, $ h_{23}, h_{145}\nsubseteq \textup{supp}(\omega)$.\\ 
    
    \noindent
    \textbf{Claim}: $a_{145} \le 1.$\\

    \noindent
    Remark that $ \left( X, \frac{1}{3}\left( 3H_{145} + H_{246}+H_{356}+H_{23}+E_1+\tilde{E}_4+\tilde{E}_5 \right) \right)  $ is log canonical at $P$.
    By \cite[Lemma 2.2]{MR4216579}, we can assume that $ \support\left( D\right)  $ does not contain at least one component of $ 3H_{145} + H_{246}+H_{356}+H_{23}+E_1+\tilde{E}_4+\tilde{E}_5$ and $(X,\frac{1}{3}D)$ is not log canonical at $P$. Note that $\support\left( D\right)$  contains both $E_1$ and $H_{145}$, since otherwise it leads to a contradiction   \begin{align*}2&=E_1\cdot D\geq \mult_P(E_1)\cdot \mult_P(D)=\mult_P(D)>3 \\
    (\textup{or } 0&=H_{145}\cdot D\geq \mult_P(H_{145})\cdot \mult_P(D)=\mult_P(D)>3).
    \end{align*} 
    Therefore, 
    $$\tilde{E}_4 \not\subseteq \support\left( D\right), \tilde{E}_5 \not\subseteq \support\left( D\right), H_{23} \not\subseteq \support\left( D\right), H_{246} \not\subseteq \support\left( D\right), \textrm{or}\  H_{356} \not\subseteq \support\left( D\right).$$
    In particular, $$ e_4\not\subseteq \support\left( d\right),\ e_5\not\subseteq \support\left( d\right),\ h_{23} \not\subseteq \support\left( d\right),\  h_{246} \not\subseteq \support\left( d\right),\  \textrm{or}\   h_{356} \not\subseteq \support\left( d\right).$$
    
    \noindent
    If $e_{i} \not\subseteq \support\left( d\right)$ for $i=4, 5$, then 
    $1 =e_{i}\cdot d \geq e_{i}\cdot a_{145}h_{145}=a_{145}.$\\

    \noindent
    If $h_{23}\not\subseteq \support\left( d\right)$, then $1=h_{23}\cdot d\geq h_{23}\cdot a_{145}h_{145}= a_{145}$.\\

    \noindent
    If $h_{246}\not\subseteq \support\left( d\right)$, then $0=h_{246}\cdot d\geq h_{246}\cdot a_4e_4=a_4$. In particular, $1=e_4\cdot d\geq a_{145}$.\\

\noindent
    Similarly, if $h_{356}\not\subseteq \support\left( d\right)$, then $0=h_{356}\cdot d\geq h_{356}\cdot a_5e_5=a_5$. One has $1=e_5\cdot d\geq a_{145}$.\\ 

    \noindent
    Therefore, $a_{145} \le 1.$\\
    
    \noindent
    Denote by $m = \mult_{\varphi(P)}(\omega)$.\\
    
    \noindent
    \textbf{Claim}: $a_1\leq \frac{4}{3},\ a_1+m\leq 2$ and $a_1\leq 1+a_{1k}$ for $k=2,3,6$.\\
    
    \noindent Note that  
    \begin{align*}
        1&=d\cdot h_{16}\geq a_1-a_{16}+a_{23},\\
        1&=d\cdot h_{23}\geq a_{16}+a_{145}-a_{23},\\
        0&=d\cdot h_{145}\geq a_1+a_{23}-2a_{145}.
    \end{align*}
    This yields 
    \[
        2=d\cdot (h_{16}+h_{23})\geq a_1+a_{145}\geq \frac{3}{2}a_1+\frac{1}{2}a_{23}\geq \frac{3}{2}a_1.\]
        Thus one has $a_1\leq \frac{4}{3}$.
        For $k=2,3,6$, one has $1=d\cdot h_{1k}\geq (a_1e_1+a_{1k}h_{1k})\cdot h_{1k}=a_1-a_{1k}$.\\

    \noindent Since $1\geq a_{145}$ and 
    \begin{align*}
		-a_4-a_5-a_{23}+2a_{145}   
        = (a_{1}e_{1}+\omega)\cdot h_{145} &\geq \mult_{\varphi(P)}\left( \left.(a_{1}e_{1}+\omega)\right|_{h_{145}} \right)\\
        &= a_1+m,
    \end{align*}
    we have $a_1+m\leq 2$.\\

        \noindent
    Let $\pi: Y \longrightarrow Y_6$ be the blowing-up of $Y_6$ at $\varphi(P)$. We denote by $F$ the exceptional divisor, by $e_1^1$, $h_{145}^1$, $\omega_1$, $d_1$ the proper transform of $e_1$, $h_{145}$, $\omega$, $d$, respectively, and by $e_4^1, \ e_5^1,\ h_{12}^1,\ h^1_{13}$, $h^1_{16}$, $h^1_{23}$ the pull-back of $e_4, \ e_5,\ h_{12},\ h_{13}$, $h_{16}$, $h_{23}$, respectively. So we have that
    \begin{align*}
        d_1 &= a_1e_1^1+a_4e_4^1+a_5e_5^1+a_{12}h_{12}^1+a_{13}h_{13}^1+a_{16}h_{16}^1+a_{23}h_{23}^1\\
        &\ \ \ \ +a_{145}h_{145}^1+ \omega_1
    \end{align*}

    \noindent
    and
    $$d_1 = \pi^{*}(d) - (a_1 + a_{145} + m)F,\ e^1_1 = \pi^{*}(e_1) - F,\ h^1_{145} = \pi^{*}(h_{145}) - F,\ \omega_1=\pi^{*}(\omega)-mF.$$
    \noindent
    By \cite[Remark 2.6]{MR2465686}, $\left( Y, \frac{1}{3}d_1 +\frac{1}{2}e_1^1+\frac{1}{2}h_{145}^1+\frac{a_1 + a_{145} + m}{3}F\right)  $ is not log canonical at some point $Q \in F$.\\

    \noindent
    Suppose that $Q \in F\setminus (e_1^1 \cup h_{145}^1)$.\\

    \noindent
    So $\left( Y, \frac{1}{3}\omega_1 +\frac{a_1 + a_{145} + m}{3}F\right)  $ is not log canonical at $ Q $. Since $$\frac{a_1 + a_{145} + m}{3} \le 1,$$
    \noindent
    the pair $\left( Y, \frac{1}{3}\omega_1 +F\right)  $ is not log canonical at $ Q $. By the inversion of adjunction formula, we get $ \left( F, \left.\frac{1}{3}\omega_1\right|_{F}  \right)  $ is not log canonical at $ Q $. 
    This implies that
    \begin{align*}
        3 < F\cdot  \omega_1=F\cdot (d_1-a_1e_1^1-a_{145}h_{145}^1) 
        =(a_1 + a_{145} + m) - a_1 -a_{145} = m.
    \end{align*}

    \noindent
    Thus we get $m >3$ which contradicts to the above inequalities $a_1 + m \leq 2$ and $a_1\geq 0$.\\

    \noindent
    Suppose that $Q \in F \cap h_{145}^1$.\\

    \noindent
    So $\left( Y, \left(\frac{a_{145}}{3}+\frac{1}{2}\right)h^1_{145}+\frac{1}{3}\omega_1 +\frac{a_1 + a_{145} + m}{3}F\right)  $ is not log canonical at $ Q $. Since $\frac{a_{145} }{3}+\frac{1}{2} < 1,$ the pair $\left( Y, h^1_{145}+\frac{1}{3}\omega_1 +\frac{a_1 + a_{145} + m}{3}F\right)  $ is not log canonical at $ Q $. By the inversion of adjunction formula, we get $ \left( h^1_{145}, \left.\left(\frac{1}{3}\omega_1+\frac{a_1 + a_{145} + m}{3}F\right)\right|_{h^1_{145}}  \right)  $ is not log canonical at $ Q $. 
    This implies that
    \begin{align*}
        3 &<h^1_{145} \cdot ( \omega_1 + (a_1 + a_{145} + m)F)=h^1_{145} \cdot  \omega_1 + (a_1 + a_{145} + m)\\
        &=h^1_{145} \cdot(d_1-a_1e_1^1-a_4e_4^1-a_5e_5^1-a_{23}h_{23}^1-a_{145}h_{145}^1) + (a_1 + a_{145} + m)\\
        &=(0 - (a_1 + a_{145} + m))  -a_{4}-a_5-a_{23} +3a_{145}+ (a_1 + a_{145} + m).
    \end{align*}

    \noindent
    Thus we get $3+a_4+a_5+a_{23}< 3a_{145}$ which contradicts to the inequality $a_{145} \leq 1$.\\

    \noindent
    Suppose that $Q \in F \cap e^1_{1}$.\\

    \noindent
    So $\left( Y, \left(\frac{1}{2}+\frac{a_{1}}{3}\right)e^1_{1}+\frac{1}{3}\omega_1 +\frac{a_1 + a_{145} + m}{3}F\right)  $ is not log canonical at $ Q $. Since $\frac{1}{2}+\frac{a_{1} }{3} < 1,$ the pair $\left( Y, e^1_{1}+\frac{1}{3}\omega_1 +\frac{a_1 + a_{145} + m}{3}F\right)  $ is not log canonical at $ Q $. By the inversion of adjunction formula, we get $ \left( e^1_{1}, \left.\left(\frac{1}{3}\omega_1+\frac{a_1 + a_{145} + m}{3}F\right)\right|_{e^1_{1}}  \right)  $ is not log canonical at $ Q $. 
    This implies that
    \begin{align*}
        3 &<e^1_{1} \cdot ( \omega_1 + (a_1 + a_{145} + m)F)\\ 
        &= e^1_{1}\cdot (d_1 - a_1e^1_1 - a_{12}h^1_{12} - a_{13}h^1_{13} - a_{16}h^1_{16}-a_{145}h^1_{145} ) + (a_1 + a_{145} + m)\\
        &=(1 - (a_1 + a_{145} + m)) + 2a_{1} - ( a_{12} + a_{13} + a_{16}+0 ) + (a_1 + a_{145} + m).
    \end{align*}

    \noindent
    Thus we get 
    $$2 +  a_{12} + a_{13} + a_{16}  < 2a_{1}.$$ 
    
    \noindent
    Together with the inequalities $a_{1} \leq 1 + a_{1k}$ for $k=2,3,6$, we achieve that
    $$2 + 3a_1 - 3 \leq 2 +  a_{12} + a_{13} + a_{16}  < 2a_{1}.$$

    \noindent
    So we get $a_1 < 1$ which contradicts to the inequalities $2\leq 2 +  a_{12} + a_{13} + a_{16} < 2a_{1}$.\\

   \noindent The same proof works for $\varphi\left( P\right) \in e_{2} \cap h_{246}$ and  $\varphi\left( P\right) \in e_{3} \cap h_{356}$.\\
   
    \noindent
    \textbf{Case 2.10}: $\varphi\left( P\right) \in h_{12} \cap h_{356}$ or $\varphi\left( P\right) \in h_{13} \cap h_{246}$ or $\varphi\left( P\right) \in h_{23} \cap h_{145}$.\\

    \noindent The arguments are similar to Case 2.9.   
\end{proof}

 \begin{Acknowledgments}
This research is funded by Vietnam Ministry of Education and Training (MOET) under grant number B2025-CTT-04. The second author was partially supported by National Science and Technology Council of Taiwan (Grant Numbers: 112-2115-M-008 -006 -MY2 and 113-2123-M-002-019-). The third author was supported by Basic Science Research Program through the National Research Foundation of Korea(NRF) funded by the Ministry of Education(No. RS-2023-00241086). The first and the second authors would like to thank Vietnam Institute for Advanced Studies in Mathematics (VIASM) for their hospitality.

\end{Acknowledgments}

\bibliographystyle{acm}
\bibliography{JhengJie-Bin-YongJoo}

\begin{paracol}{2}
\BinAddresses
  \switchcolumn
\JhengJieAddresses
  \switchcolumn
\YongJooAddresses
  \switchcolumn
\end{paracol}

\end{document}